\documentclass[12pt]{amsart}

\usepackage {amsmath}
\usepackage[utf8]{inputenc}
\usepackage[english]{babel}
\usepackage {amsfonts, amssymb}
\usepackage {latexsym}

\usepackage{amsthm}

\newtheorem {theorem} {Theorem} [section]
\newtheorem {lemma}[theorem] {Lemma} 
\newtheorem {prop} [theorem]{Proposition}
\newtheorem {cor} [theorem]{Corollary}

\theoremstyle{remark}
\newtheorem {remark} [theorem]{Remark}
\newtheorem {example} [theorem]{Example}

\newcommand {\R} {\mathbb{R}}
\def\fr#1#2{\frac{\partial #1}{\partial #2}}

\newcommand{\dd}{d}
\newcommand{\cd}{\, d}
\newcommand{\e}{\mathrm e}
\DeclareMathOperator{\E}{\mathsf e}
\DeclareMathOperator{\re}{\Sigma}
\DeclareMathOperator{\ep}{\mathsf I}
\begin{document}

\title[Velocity of ratchets]{On the bulk velocity of Brownian ratchets}

\author[S.~Kondratyev]{Stanislav Kondratyev}
\address[S.~Kondratyev]{CMUC, Department of
Mathematics, University of Coimbra, 3001-501 Coimbra, Portugal}{}
\email{kondratyev@mat.uc.pt}

\author[J.M.~Urbano]{Jos\'{e} Miguel Urbano}
\address[J.M.~Urbano]{CMUC, Department of
Mathematics, University of Coimbra, 3001-501 Coimbra, Portugal}{}
\email{jmurb@mat.uc.pt}

\author[D.~Vorotnikov]{Dmitry Vorotnikov}
\address[D.~Vorotnikov]{CMUC, Department of
Mathematics, University of Coimbra, 3001-501 Coimbra, Portugal}{}
\email{mitvorot@mat.uc.pt}

\thanks{The third author is grateful to Danielle Hilhorst, Steven Lade, Rafael 
Ortega and Fabio Zanolin for very useful correspondence. The research was 
partially supported by CMUC, funded by COMPETE and FCT under 
the project PEst-C/MAT/UI0324/2013, and by FCT projects UTA-CMU/MAT/0007/2009 
and PTDC/MAT-CAL/0749/2012.}


\numberwithin{equation}{section}


%\infonum{10}{03}

%\date{January 22, 2010}

\begin{abstract}
In this paper we study the unidirectional transport effect for Brownian 
ratchets modeled by Fokker--Planck-type equations.  In particular, we consider 
the adiabatic and semiadiabatic limits for tilting ratchets, generic ratchets 
with small diffusion, and the multi-state chemical ratchets. Having 
established a linear relation between the bulk transport velocity and the 
bi-periodic solution, and using relative entropy estimates and new functional 
inequalities, we obtain explicit asymptotic formulas for the transport 
velocity and qualitative results concerning the direction of transport. In 
particular, we prove the conjecture by Blanchet, Dolbeault and Kowalczyk 
that the bulk velocity of the stochastic Stokes' drift is non-zero for every 
non-constant potential.
\end{abstract}

\maketitle

\section{Introduction}
\label{intro}

\emph{Brownian ratchets} or \emph{Brownian motors} are generic terms for tiny devices which are able to produce unidirectional transport of matter when all
acting forces and gradients vanish after averaging over space and time, and at 
the presence of (and often due to) overdamped Brownian motion 
\cite{as97,fi02,hmn05,par02,rei02,rh02}.  Motor proteins are generally 
considered to be the most celebrated example of Brownian ratchets \cite{ah03}.  
However, during recent years there has been a huge progress \cite{hm09} in 
realizing and observing bulk motion without net bias in SQUIDs, Josephson 
junctions, cold atoms in optical lattices, nanopores, etc., as well as in 
microfluidics experiments.  The Stokes' drift with diffusion is also an 
example of a Brownian ratchet mechanism \cite{bena00,hmn05,jl98}.  Although 
the idea of micro-level motors goes back to the dawn of thermodynamics, the 
discovery of ratchets has boosted the contemporary nano-technological interest 
in the development of hybrid and artificial molecular machines 
\cite{bf06,klz07}.

The dynamics of a ratchet can be described by the Fokker--Planck equation
\begin{equation} \label{eq:int-general-ratchet}
\rho_t - \sigma \rho_{xx} - (\Psi_x \rho)_x = 0, 
\end{equation}
where $\sigma$ is the given diffusion coefficient, and $\rho(x, t)$ is the 
unknown probability density of distribution of Brownian particles governed by 
a given potential $\Psi(x, t)$, which is supposed to be $T$-periodic in time 
$t$, and to have a $1$-periodic in $x$ derivative $\Psi_x(x, t)$.  Note that 
we do not assume the potential $\Psi(x, t)$ itself to be $x$-periodic, so 
various tilting regimes are allowed, and the `tilting forces' are contained 
within the potential. 

A related model which has particular relevance in biological applications is 
the chemical motor.  Here the particles can be in several states, and the 
total amount of particles is fixed.  Particles in different states are 
sensitive to different time-independent potentials.  The underlying chemical 
processes  cause transitions between the particles' states, which we can 
consider to be random.  This is described by the following system of 
Fokker--Planck-type equations:
\begin{equation}
\label{eq:int-general-system}
(\rho_i)_t - \sigma(\rho_i)_{xx} - ((\Psi_i)_x \rho_i)_x + \sum\limits_{j,\ j\neq i}^{}\nu_{ji} \rho_i  = \sum\limits_{j,\ j\neq i}^{}\nu_{ij} \rho_j, \ i=1,\dots, N,
\end{equation}
where $\Psi_i(x)$ are the given potentials, and $\sigma$ is the diffusion 
coefficient (for definiteness, we set it to be the same for all states).  We 
assume that $(\Psi_i)_x(x)$ (not the $\Psi_i$ themselves) and the transition 
rates $\nu_{ij}(x)$ are $1$-periodic. 

Various particular cases of \eqref{eq:int-general-ratchet} and 
\eqref{eq:int-general-system}, including the so-called \emph{flashing 
ratchets}, were studied in \cite{bdk08, bdk09, chk04, chko09, dkk04, hkl07, 
kk02,ms13,ps09a,ps09,ps11,vor11,vor13}.  To catch the motor effect, the 
majority of these papers consider  \eqref{eq:int-general-ratchet} or 
\eqref{eq:int-general-system} with no-flux boundary conditions on a bounded 
segment, and show, under appropriate assumptions, that the mass is eventually 
concentrated closer to one edge of the segment than to the other.  Yet 
equation \eqref{eq:int-general-ratchet} with the \emph{travelling} potential 
$\Psi(x, t) = \psi(x-\omega t)$ and the \emph{flashing} potential $\Psi(x, t) 
= h(t)\psi(x)$, where $\psi$ is $1$-periodic, and
\begin{equation*}
h(t)=
\begin{cases}
1 & \text{ if } kT < t \le (k+1/2)T , \\
0 & \text{ if } (k + 1/2)T < t \le (k+1)T,
\end{cases}
\quad (k = 0, 1, \dots)
\end{equation*}
was examined on the whole real line in \cite{bdk08, bdk09} and \cite{ms13}, 
respectively.  In \cite{ms13} it was observed that the solutions of a 
homogenized equation tend to propagate with constant speed or not to move at 
all as the period length goes to zero.  In \cite{bdk08, bdk09} it was shown 
that with the course of time the velocity of the centre of mass eventually 
becomes the same for all solutions (see also \cite{da11}).  Moreover, this 
asymptotic speed is equal to \begin{equation}
\label{eq:int-def-asymptotic}
v_\infty
=
- \frac 1T
\int_0^T
\int_{0}^{1}
\Psi_x(x, t) g_\infty(x, t)
\cd x
\cd t
,
\end{equation} where $g_\infty$ is the bi-periodic (in $x$ and $t$) travelling wave solution to \eqref{eq:int-general-ratchet}. 

In this paper, we develop a unified approach for detecting transport for 
generic equations \eqref{eq:int-general-ratchet} and 
\eqref{eq:int-general-system}.  We prove that the averaged velocity stabilizes 
as time goes to infinity, and the limiting velocity is independent of the 
solution.  We establish a linear relation between this velocity and a certain 
solution to \eqref{eq:int-general-ratchet} or \eqref{eq:int-general-system}, 
respectively.  This solution is actually the bi-periodic solution in the case 
of \eqref{eq:int-general-ratchet}, and is the stationary $x$-periodic solution 
vector in the case of \eqref{eq:int-general-system}.  This allows us to obtain 
a more explicit characterization of the occurrence of unidirectional transport, its 
direction and bulk velocity, for $1$- and $2$-state tilting ratchets, in the 
adiabatic and semiadiabatic regimes, for the stochastic Stokes' drift, and for 
generic low-diffusion-driven $1$- and $2$-state ratchets 
\eqref{eq:int-general-ratchet} and \eqref{eq:int-general-system}.  We also prove the 
conjecture stated by Blanchet, Dolbeault, and 
Kowalczyk in~\cite{bdk08} (see also \cite{rei02}) that the bulk velocity of 
the stochastic Stokes' drift is non-zero for every non-constant potential.

The paper is organized as follows.  In Section~\ref{sectionbv} we set the 
framework for our research.  In particular, we define the transport in terms 
of the asymptotic average bulk velocity and relate it to the bi-periodic 
solution of an auxiliary space-periodic problem.

In Section~\ref{altr} we consider the adiabatic regime for tilting ratchets, 
i.e., we suppose that the ratchet spends a long time in each of its states.  
In Theorem~\ref{th:adiabatic-limit} we justify the explicit 
formula~\eqref{eq:adiab-oooo} for the adiabatic transport velocity.  
Developing this topic, we state an effective formula for the direction of 
transport (Proposition~\ref{pr:shsh}).  Finally, Theorem~\ref{th:adiab-foo} 
gives a qualitative result showing that a major interval of monotonicity of 
the potential implies a particular direction of transport.  A highlight of 
Section~\ref{altr} is Proposition~\ref{pr:phiphi} establishing a nontrivial 
functional inequality.

In Section~\ref{slssd} we study the so-called semiadiabatic regime for tilting ratchets, when the 
overall period of tilting goes to infinity and one of the tilting states 
dominates the other.  We give an effective explicit 
formula~\eqref{eq:semiad-limit-velocity} for the semiadiabatic transport 
velocity and prove that non-constant potentials produce nonzero semiadiabatic 
transport in one and the same direction---Theorem~\ref{th:semiad-velocity} and 
Corollaries \ref{cor:semiad-lim}~and \ref{cor:semiad-pos}.  These results are partially based on the functional inequality related with Proposition~\ref{pr:funct}, which also implies the
conjecture of~\cite{bdk08} (see also \cite{rei02}) that the stochastic Stokes' drift generates unidirectional transport for every non-constant potential. 

In Section~\ref{sdc} we consider generic Brownian ratchets with small diffusion 
coefficient and show that there is a connection between the transport and a certain ODE.  
We show (Theorem~\ref{odeth}) that if this ODE does not have periodic solutions 
or, equivalently, if its Poincar\'e rotation number in nonzero, there appears 
directed transport of mass.

In Section~\ref{msm} we extend our approach to multi-state models.  The most 
interesting results are obtained for the case of small diffusion: 
Theorem~\ref{odesys} classifies it with respect to the geometry of zeroes of 
the potential gradients and establishes the direction of transport in 
different cases.  Finally, Theorem~\ref{rtasa} treats the adiabatic and semiadiabatic regimes for the 
randomly tilting ratchet.

\section {Bulk velocity}
\label{sectionbv}

\subsection{Unidirectional transport}

We model the dynamics of a ratchet by the Fokker--Planck equation
\begin{equation}
\label{eq:general-ratchet}
\left\{
\begin{array}{l}
\rho_t - \sigma \rho_{xx} - (\Psi_x \rho)_x = 0, \quad x \in \R,\ t > 0, \\
\rho = \rho_0(x), \quad x \in \R.
\end{array}
\right.
\end{equation}
Here $\Psi(x, t)$ is a given potential, and $\rho_0$ is a given initial 
condition.  We assume that $\Psi(x, t)$ is $T$-periodic in $t$ and its 
derivative $\Psi_x(x, t)$ is $1$-periodic in $x$, and we also assume that 
$\rho_0(x)$ satisfies the requirements
\begin{equation}
\label{eq:ic-requirements}
\rho_0(x) \ge 0,\ \int_{-\infty}^\infty \rho_0 (x) \cd x = 1,\ \int_{-\infty}^\infty |x| \rho_0(x) \cd x < 
\infty.
\end{equation}
We are interested in nonnegative solutions of \eqref{eq:general-ratchet}.  
Such solutions may be viewed as non-stationary distributions of a unit mass on 
$\R$ whose movement is governed by a potential force and by diffusion.

The most interesting cases arise when $\Psi$ is `unbiased' in the sense that
\begin{equation}
\label{eq:unbiasedPsi}
\int_0^T \int_0^1 \Psi_x(x,t) \cd x \cd t = 0,
\end{equation}
i.e., the time and space average of the potential gradient vanishes.  However, 
for the sake of generality, in the sequel we do not assume 
\eqref{eq:unbiasedPsi} unless explicitly specified.

It can be proved by classical methods (see e.g.~\cite{pe84}) that if $\Psi$ is 
continuous in $(x, t)$ and $C^{2,\alpha}$-regular in $x$, where $\alpha \in 
(0, 1)$ is independent of $t$, then \eqref{eq:general-ratchet} is uniquely 
solvable for any continuous initial data $\rho_0$ satisfying 
\eqref{eq:ic-requirements}; moreover, the solution $\rho(x, t)$ is positive 
for any $t > 0$, and
\begin{equation}
\label{eq:limit}
\lim_{x \to \pm\infty}
|x|(|\rho(x, t)|+|\rho_x(x,t)|)
= 0
.
\end{equation}
A consequence of \eqref{eq:limit} is the \emph{conservation of mass}
\begin{equation}
\label{eq:mass-conservation}
\int_{-\infty}^\infty \rho \, \dd x = 1
\end{equation}
and the finiteness of the \emph{centre of mass}
\begin{equation}
\label{eq:centre-of-mass}
\bar x(t)
=
\int_{-\infty}^{\infty}
x \rho(x, t)
\,\dd x
\end{equation}
for any $t$.  We study the asymptotic behaviour of solutions of 
\eqref{eq:general-ratchet} as $t \to \infty$.  Properties \eqref{eq:limit}, 
\eqref{eq:mass-conservation}, and \eqref{eq:centre-of-mass} are crucial for 
our approach.

To catch the unidirectional transport effect, we consider the velocity of the 
centre of mass $\bar x(t)$, which is called the \emph{drift} (or \emph{bulk}, 
or \emph{ballistic}) \emph{velocity}.  Specifically, we consider the average 
drift velocity on the interval $[t, t + T]$, and if it has a nonzero limit as 
$t \to +\infty$, we say we have unidirectional transport.

Due to the periodic nature of the problem at issue, the drift velocity is 
conveniently characterized by means of the following problem on the circle 
$S^1 = \R / \mathbb Z$:
\begin{gather}
\label{eq:per-pr}
g_t - \sigma g_{xx} + (F g)_x = 0, \quad (x, t) \in S^1 \times (0, \infty),
\\
\label{eq:per-pr-ic}
g(x, 0) = g_0(x), \quad x \in S^1; \quad
g_0(x) \ge 0, \ \int_{S^1}g_0(x) \cd x = 1
.
\end{gather}
Here $F(x, t)$ is defined on $S^1 \times (0, \infty)$; generally we assume 
that it is $T$-periodic in $t$. 

To start with, observe that if $\rho$ is a nonnegative solution 
of~\eqref{eq:general-ratchet}, then, by linearity, the nonnegative function
\begin{equation}
\label{eq:rho-g}
g(x, t)
=
\sum_{k = -\infty}^\infty
\rho(x + k, t)
\end{equation}
solves~\eqref{eq:per-pr} with $F = - \Psi_x$.

Consider the average velocity of the centre of mass as $t$ varies from $t_0$ 
to $t_0 + T$:
\begin{multline*}
v_{[t_0,t_0+T]} :=
\frac{
\bar x(t_0 + T) - \bar x(t_0)
}{
T
}
\\
=
\frac 1T
\left(
\int_{-\infty}^{\infty}
x \rho(x, t_0 + T) \cd x
-
\int_{-\infty}^{\infty}
x \rho(x, t_0) \cd x
\right)
\\
=
\frac 1T
\int_{-\infty}^{\infty}
x
\int_{t_0}^{t_0+T} \rho_{t} (x, t) \cd t
\cd x
\\
=
\frac 1T
\int_{t_0}^{t_0+T}
\int_{-\infty}^{\infty}
x (\sigma \rho_{xx} + (\Psi_x \rho)_x)
\cd x
\cd t
\\
=
\frac 1T
\int_{t_0}^{t_0+T}
\left(
- \sigma \int_{-\infty}^{\infty} \rho_{x} 
\cd x
-
\int_{-\infty}^{\infty}
 \Psi_x \rho
\cd x
\right)
\dd t
\\
=
-
\frac 1T
\int_{t_0}^{t_0+T}
\sum_{k=-\infty}^{\infty}
\int_{0}^{1}
\Psi_x(x + k, t) \rho(x + k, t)
\cd x
\cd t
\\
=
- \frac 1T
\int_{t_0}^{t_0+T}
\int_{0}^{1}
\Psi_x(x, t) g(x, t)
\cd x
\cd t
\end{multline*}
(here we have used \eqref{eq:limit} and the periodicity of $\Psi_x$).  Thus we 
have the following formula for the average drift velocity:
\begin{equation}
\label{eq:av-vel}
v_{[t_0, t_0+T]} =
- \frac 1T
\int_{t_0}^{t_0+T}
\int_{0}^{1}
\Psi_x(x, t) g(x, t)
\,\dd x
\,\dd t
.
\end{equation}

\begin{remark}
We point out that the conservation of mass holds for 
\eqref{eq:per-pr}--\eqref{eq:per-pr-ic}, i.e.,
\begin{equation}
\label{eq:per-mass-conservation}
\int_{S^1} g(x, t) \cd x = 1
\end{equation}
for any solution thereof.
\end{remark}

\subsection{Bi-periodic solution}

The notion of relative entropy and related inequalities are a useful tool for 
the study of the Fokker--Planck equation~\eqref{eq:per-pr}.

Given $g, h \in L^1_+(S^1)$ such that
\begin{equation*}
\int_{S^1}
g
\cd x
=
1
=
\int_{S^1}
h
\cd x
,
\end{equation*}
define the \emph{relative entropy} of $g$ with respect to $h$ by
\begin{equation*}
\E [ g | h ] =
\int_{S^1}
g \ln \frac gh
\cd x
.
\end{equation*}
Observe that $\E [g | h] \ge 0$ (the possibility $\E [g | h ] = \infty$ is 
not excluded).  Indeed, letting $r = g/h \ge 0$ we have
\begin{multline*}
\E [g | h]
=
\int_{S^1}
g \ln r
\cd x
=
\int_{S^1}
(g \ln r - g + h)
\cd x
\\
=
\int_{S^1}
h (r \ln r - r + 1)
\cd x
\ge 0
\end{multline*}
as $r \ln r - r + 1 = \int_1^r \ln \xi \cd \xi \ge 0$.  Moreover, $\E[g|h] = 
0$ if and only if $g = h$ almost everywhere.  This follows from the definition 
and from the fact that the relative entropy controls $L^1$ distance.  
Specifically, for probability densities $g, h \in L_+^1(S^1)$ we have the 
known \emph{Csisz\'{a}r--Kullback inequality} \cite{csi75}:
\begin{equation}
\label{eq:cs-k}
\| g - h \|_{L^1(S^1)}^2 \le 2 \E [ f | g ]
,
\end{equation}
which holds for any probability densities $g, h \in L_+^1(S^1)$.

Another important tool is the \emph{Log-Sobolev inequality} \cite{go09}, 
which we need in the following form: given $h \in L^1(S^1)$ such that
\begin{equation*}
\int_{S^1}
h
\cd x
=1,\quad 0 < C_1 \le h(x) \le C_2 \ \text{ for a. a. } x \in S^1
,
\end{equation*}
there exists $b = b(C_1, C_2) > 0$ such that
\begin{equation}
\label{eq:log-sobolev}
\int_{S^1}
g \ln \frac gh
\cd x
\le
b
\int_{S^1}
g
\left|
\left(
\ln \frac gh
\right)_x
\right|^2
\cd x
\end{equation}
for any probability density $g \in L_+^1(S^1)$ such that the derivative on the 
right-hand side exists almost everywhere.  Observe that both sides 
in~\eqref{eq:log-sobolev} are nonnegative (possibly infinite) whenever the 
derivative of the logarithm makes sense.

The integral
\begin{equation*}
\ep [ g | h ]
=
\int_{S^1}
g
\left|
\left(
\ln \frac gh
\right)_x
\right|^2
\cd x
\end{equation*}
on the right-hand side of~\eqref{eq:log-sobolev} is called the \emph{entropy 
production term} of $g$ and $h$.  Thus~\eqref{eq:log-sobolev} can be 
equivalently expressed as
\begin{equation}
\label{eq:varlog-sobolev}
\E [g | h]
\le
b
\ep [ g | h]
.
\end{equation}

The following theorem sets the framework for our investigation of transport.  
It establishes the existence of a unique time-periodic solution 
of~\eqref{eq:per-pr}, which attracts other solutions and thus is our natural 
object of study (cf.~\cite{dkk04, bdk09, bdik07}).

\begin{theorem}
\label{th:biper-sol}
Suppose that $F\colon S^1 \times \R_+$ is $T$-periodic in the second argument 
and there exists a partition $0 = t_0 < t_1 < \dots < t_n = T$ of the segment 
$[0,T]$ such that $F$ is $C^3$ on every segment $[t_{i-1}, t_{i}]$.  Then 
there exists a unique positive $T$-periodic in $t$ solution $g_\infty$ 
of~\eqref{eq:per-pr} satisfying~\eqref{eq:per-mass-conservation}.  Moreover, 
if $g$ solves~\eqref{eq:per-pr}--\eqref{eq:per-pr-ic} with the initial 
condition satisfying $\int_{S_1} g_0 \ln g_0 \cd x < \infty$, then $g_\infty$ 
attracts $g$ in the sense of the entropy
\begin{equation}
\label{eq:biper-entropy-attraction}
\E[ g(\cdot, t) | g_\infty(\cdot, t) ]
\le
\E[ g(\cdot, 0) | g_\infty(\cdot, 0) ]
\e^{- \gamma t},
\ t \ge 0,
\end{equation}
where $\gamma = \sigma/b$ and $b = b(\min g_\infty, \max g_\infty)$ is the 
uniform Log-Sobolev constant for $g_\infty$.
\end{theorem}
\begin{remark}
If $F$ has discontinuities at the points $t_i$, we do not require that 
\eqref{eq:per-pr} should hold at these points.  In this case we construct 
solutions piecewise: given the initial data at $t = t_0$, by parabolic 
regularity the solution is well defined at $t = t_1$, then $g(\cdot, t_1)$ is 
considered as the initial condition for the segment $[t_1, t_2]$, and so on.
\end{remark}
\begin{proof}[Proof of Theorem~\ref{th:biper-sol} ]
First suppose that $F$ is $C^3$ on $S^1 \times [0,T]$.
Let
$$X = \{ g \in L^1_+(S^1) \mid \int_{S^1} g(x) \cd x = 1\}.$$
Consider the operator $\mathcal T \colon X \to X$ that takes $g_0 \in X$ to 
$g(\cdot, T)$, where $g$ solves~\eqref{eq:per-pr}--\eqref{eq:per-pr-ic}.  The 
operator is well defined due to the conservation of mass and the maximum 
principle.  As $X$ is convex and bounded in $L^1(S^1)$, its image $\mathcal 
T(X)$ is precompact in $X$ by parabolic regularity.  Hence, by the Schauder 
theorem, $\mathcal T$ has a fixed point, which clearly is the initial 
condition for a $T$-periodic in time solution $g_\infty$ of~\eqref{eq:per-pr}.  
We have yet to prove that it is the only periodic solution.  By parabolic 
regularity, $g_\infty$ is continuous on $S^1 \times [0,T]$ and hence bounded, 
and by the strong maximum principle it is positive and thus bounded away from 
0.

Now suppose $g$ solves~\eqref{eq:per-pr} with the initial condition $g_0(x) 
\ge 0$ such that $\int_{S^1} g_0(x) \ln g_0(x) \cd x < \infty$.  We claim that 
the relative entropy $\E[ g(\cdot, t) | g_\infty(\cdot, t)] $ decreases. 
Indeed, letting $r = g/g_\infty$ we have
\begin{multline}
\label{eq:biper-sol-1}
\frac{\dd \E[g | g_\infty]}{\dd t}
=
\frac{\dd}{\dd t}
\int_{S^1}
g \ln \frac{g}{g_\infty}
\cd x
\\
=
\int_{S^1}
g_t
\ln \frac{g}{g_\infty}
\cd x
+
\int_{S^1}
g_t
\cd x
-
\int_{S^1}
\frac{
g (g_\infty)_t
}{
g_\infty
}
\cd x
\\
=
\int_{S^1}
(\sigma g_{xx} - (Fg)_x) \ln r
\cd x
-
\int_{S^1}
(g_\infty)_t r
\cd x
\\
=
\int_{S^1}
\frac{g}{r}
(\sigma r_{xx} + F r_x)
\cd x
-
\sigma
\int_{S^1}
g
\frac{r_x^2}{r^2}
\cd x
-
\int_{S^1}
(g_\infty)_t r
\cd x
\\
=
\int_{S^1}
g_\infty
(\sigma r_{xx} + F r_x)
\cd x
-
\sigma
\int_{S^1}
g
\frac{r_x^2}{r^2}
\cd x
-
\int_{S^1}
(g_\infty)_t r
\cd x
\\
=
\int_{S^1}
(\sigma (g_\infty)_{xx} - (Fg_\infty)_x - (g_\infty)_t)
r
\cd x
-
\sigma
\int_{S^1}
g
\frac{r_x^2}{r^2}
\cd x
\\
=
-
\sigma
\int_{S^1}
\left( \ln \frac{g}{g_\infty} \right)_x^2 g
\cd x
=
- \sigma \ep [ g | g_\infty]
.
\end{multline}

As $g_\infty$ is bounded away from 0, as well as from above for $(x, t) \in 
S^1 \times [0,T]$, we have the uniform in $t$ Log-Sobolev inequality
\begin{equation*}
\E [ g | g_\infty] \le b \ep [ g | g_\infty]
.
\end{equation*}
Hence
\begin{equation}
\label{eq:biper-sol-2}
\frac{\dd \E [ g | g_\infty]}{\dd t}
\le
- \frac \sigma b \E[ g | g_\infty]
\end{equation}
and~\eqref{eq:biper-entropy-attraction} follows.  Moreover, the 
attraction~\eqref{eq:biper-entropy-attraction} implies the uniqueness of the 
periodic solution.

In the general case of piecewise continuous $F$, the proof works with slight 
modifications.  We have $\mathcal T = \mathcal T_n \circ \dots \circ \mathcal 
T_1$, where $\mathcal T_i$ is the resolving operator for the segment 
$[t_{i-1}, t_i]$, which is compact, so we still can apply the Schauder theorem 
and obtain a fixed point and a periodic solution.  The relative entropy $\E[g 
| g_\infty]$ is continuous and the computation \eqref{eq:biper-sol-1} holds in 
each open interval $(t_{i-1}, t_i)$, so $\E[g | g_\infty]$ decreases.  
Inequality~\eqref{eq:biper-sol-2} also holds in each open interval $(t_{i-1}, 
t_i)$, whence \eqref{eq:biper-entropy-attraction} follows.  Indeed, for $t_1$ 
we have
\begin{equation*}
\E[g | g_\infty]|_{t = t_1} \le \E[g | g_\infty]|_{t = t_0} \e^{- \gamma t_1}
,
\end{equation*}
so for any $t \in (t_1, t_2)$ we have
\begin{equation*}
\E
\le \E|_{t = t_1} \e^{- \gamma (t - t_1)}
\le \E|_{t = 0} \e^{- \gamma t_1}
\e^{- \gamma (t - t_1)}
= \E|_{t = 0} \e^{- \gamma t}
,
\end{equation*}
and so on for the subsequent intervals.
\end{proof}

\begin{cor}
Under the hypothesis of Theorem~\ref{th:biper-sol}, $g_\infty$ also 
exponentially attracts $g$ in the sense of $L^1$:
\begin{equation}
\label{eq:biper-attraction}
\int_{S^1}
| g(x, t) - g_\infty(x, t) |\cd x \le C \e^{-\gamma t/2}
,
\ t \ge 0
\end{equation}
where $C$ depends on $g$ and $g_\infty$.
\end{cor}
\begin{proof}
It suffices to combine~\eqref{eq:biper-entropy-attraction} with the 
Csisz\'{a}r--Kullback inequality~\eqref{eq:cs-k}.
\end{proof}

Returning to problem \eqref{eq:general-ratchet}, put
\begin{equation}
\label{eq:def-asymptotic-velocity}
v_\infty
=
- \frac 1T
\int_0^T
\int_{0}^{1}
\Psi_x(x, t) g_\infty(x, t)
\cd x
\cd t
,
\end{equation}
where $g_\infty$ is the periodic solution of \eqref{eq:per-pr} with $F = - 
\Psi_x$ satisfying~\eqref{eq:per-mass-conservation}.  Combining 
\eqref{eq:av-vel} and \eqref{eq:biper-attraction}, we easily obtain the next 
result.
\begin{cor}
\label{cor:asymptotic-velocity}
Suppose that $\Psi(x,t)$ is $C^4$-smooth in $x$, and $F = - \Psi_x$ satisfies 
the hypothesis of Theorem~\ref{th:biper-sol}.  Then for any solution $\rho$ of 
\eqref{eq:general-ratchet} with the initial condition $\rho_0$ satisfying 
\eqref{eq:ic-requirements} we have
\begin{equation}
\label{eq:asymptotic-velocity}
| v_{[t_0, t_0 + T]} - v_\infty | \le C \e^{-\gamma t_0/2}
,
\end{equation}
where $\gamma > 0$ is the same as in Theorem~\ref{th:biper-sol} and depends 
only on $g_\infty$, and $C$ depends on $\rho$ and $g_\infty$.
\end{cor}

We conclude that under the hypothesis of 
Corollary~\ref{cor:asymptotic-velocity} the limiting average drift velocity is 
the same for all solutions of~\eqref{eq:general-ratchet} and is determined by 
the periodic equation~\eqref{eq:per-pr}.  For this reason in what follows we 
mostly concentrate on equation~\eqref{eq:per-pr}.

\subsection{Tilting and tilted ratchets}

An important class of unbiased potentials are the \emph{tilting potentials}, 
which have the form
\begin{equation}
\label{eq:general-tilt}
\Psi(x, t) = \psi(x) + H(t) x
,
\end{equation}
where the base potential $\psi(x)$ is $1$-periodic in $x$, and $H(t)$ is 
$T$-periodic in $t$ and characterized by the property
\begin{equation}
\label{eq:unbiasedH}
\int_0^T H(t) \cd t = 0
.
\end{equation}
Given the periodicity of $\psi$, equation~\eqref{eq:unbiasedH} is equivalent 
to \eqref{eq:unbiasedPsi}.

We obtain a typical tilting potential by letting
\begin{equation}
\label{eq:omegah}
H(t)=h(t)\omega
,
\end{equation}
where $\omega \in \R$ characterizes the swing of the tilt, and
\begin{equation}
\label{eq:switching-h}
h(t)=
\begin{cases}
1 & \text{ if } kT < t \le (k+1/2)T, \\
-1 & \text{ if } (k + 1/2)T < t \le (k+1)T,
\end{cases}
(k = 0, 1, \dots).
\end{equation}
The tilting potential corresponding to~\eqref{eq:switching-h} periodically 
switches between the tilted potentials $\psi(x) \pm \omega x$.  Thus the 
\emph{tilted potential}
\begin{equation}
\label{eq:tilted-potential}
\psi(x) + \omega x
\end{equation}
is a useful example of an obviously `biased' potential (here $\psi$ satisfies 
the same conditions as above and $\omega$ is a number).

\begin{remark}
\label{rem:tilting-av-vel}
Observe that if $\Psi$ is a tilting ratchet potential 
\eqref{eq:general-tilt}, then
\begin{equation}
\label{eq:tilting-av-vel}
v_{[t_0, t_0+T]}
=
- \frac 1T
\int_{t_0}^{t_0+T}
\int_0^1
\psi_x (x) g (x, t)
\,\dd x
\,\dd t
.
\end{equation}
To prove this, substitute \eqref{eq:general-tilt} in \eqref{eq:av-vel} and use 
the conservation of mass \eqref{eq:per-mass-conservation} and the zero mean 
condition \eqref{eq:unbiasedH}.  Consequently,
\begin{equation}
\label{eq:tilting-asymptotic-velocity}
v_\infty
=
- \frac 1T \int_0^T
\int_{0}^{1}
\psi_x(x, t) g_\infty(x, t)
\cd x
\cd t
.
\end{equation}
\end{remark}

It is useful to consider the stationary periodic equation
\begin{equation}
\label{eq:st-per}
\sigma g_{xx} + ((\psi_x + \omega) g)_x = 0,\quad  (t, x) \in (0, \infty) 
\times S^1,
\end{equation}
where $\psi_x$ is the derivative of a $C^4$ $1$-periodic function $\psi$.  This 
equation can be treated by elementary methods.  Put
\begin{gather*}
\alpha = \alpha(\omega) = \e^\omega - 1
; \\
\beta_+ = \beta_+(\omega, \psi) = \int_0^1 \e^{\omega x + \psi(x)} \, \dd x
;\\
\beta_- = \beta_-(\omega, \psi) = \int_0^1 \e^{-\omega x - \psi(x)} \, \dd x
;\\
\beta = \beta(\omega, \psi) =
\int_0^1
\int_0^x
\e^{\omega y + \psi(y) -\omega x - \psi(x)}
\,\dd y\,\dd x
,
\end{gather*}
and define
\begin{equation}
\label{eq:defAB}
A(\omega, \psi) =
\frac{
\alpha
}{
\alpha \beta + \beta_+ \beta_-
}
, \qquad
B(\omega, \psi) =
\frac{
\beta_+
}{
\alpha \beta + \beta_+ \beta_-
}
.
\end{equation}

\begin{remark}
\label{rem:A}
Observe that $A(\omega, \psi) > 0$ for $\omega > 0$ and $A(0, \psi) = 0$.  A 
simple reflection argument
($\psi(x) \mapsto \psi(1-x)$, $\omega \mapsto - \omega$)
shows that $A(\omega, \psi) < 0$ for $\omega < 0$.  In particular, $A(\omega, 
\psi)$ exists for any $\omega$ and continuous $\psi$ (i.e.,\ the denominator 
does not vanish).  Consequently, $B(\omega, \psi)$ also exists.
\end{remark}

\begin{prop}
\label{pr:st-per}
Suppose $\psi \colon \R \to \R$ is continuously differentiable and $1$-periodic; 
then~\eqref{eq:st-per} has a unique solution $g^*$ such that
\begin{equation*}
g^*(x) \ge 0,\ \int_0^1 g^*(x) \cd x = 1.
\end{equation*}
Moreover, $g$ is given by
\begin{equation}
\label{eq:st-per-sol}
g^* =
\e^{- (\omega x + \psi(x))/\sigma}
\left(
B\left(\frac\omega\sigma, \frac\psi\sigma\right) +
A\left(\frac\omega\sigma, \frac\psi\sigma\right)
\int_0^x \e^{(\omega y + \psi(y))/\sigma}
\right)
.
\end{equation}
\end{prop}
\begin{proof}
The proof is straightforward, cf.~\cite{bdk08}.
\end{proof}

Observe that $g^*$ is the normalized time-periodic solution 
of~\eqref{eq:per-pr} with $F = - (\psi_x + \omega)$.  Applying 
Theorem~\ref{th:biper-sol}, we immediately obtain the following corollary.
\begin{cor}
\label{cor:stable-attraction}
If $g$ solves
\begin{equation*}
\left\{
\begin{array}{l}
g_t - \sigma g_{xx} - ((\psi_x + \omega)g)_x = 0, \quad x \in S^1,\ t > 0, \\
g(x, 0) = g_0(x), \quad x \in S^1,
\end{array}
\right.
,
\end{equation*}
where $g_0 \in L^1(S^1)$ is as in \eqref{eq:per-pr-ic}, and $ g_0 \ln g_0 \in L^1$, then
\begin{equation}
\label{eq:stable-attraction}
\E[ g( \cdot, t) | g^* ] \le \E[g(\cdot, 0) | g^*]\e^{-\gamma t}
,
\quad t \ge 0,
\end{equation}
where $\gamma$ only depends on the lower and upper bounds of $g^*$ and on 
$\sigma$.
\end{cor}

\begin{remark}
\label{rem:tilted-velocity}
The normalized periodic solution $g^*$ of \eqref{eq:st-per} satisfies
\begin{equation*}
\sigma g^*_x + (\psi_x + \omega) g^* = \sigma A
\end{equation*}
with $A = A(\omega/\sigma, \psi/\sigma)$.  Integrating, we get
\begin{equation}
\label{eq:A}
\sigma A = \int_0^1 (\psi_x + \omega) g^* \cd x
.
\end{equation}
Incidentally, we see that
\begin{equation}
\label{eq:tilted-velocity}
v_\infty = - \int_0^1 (\psi_x + \omega) g^* \cd x = -\sigma  A
,
\end{equation}
where $v_\infty$ is the asymptotic drift velocity for the tilted potential 
$\psi + \omega x$ (with arbitrary $T > 0$).
\end{remark}

\section {Adiabatic limit for tilting ratchets}
\label{altr}

\subsection{Asymptotic speed}

In this section we consider the tilting ratchet given by
\begin{equation} \label{eq:adiab-tilting}
\left\{
\begin{array}{l}
\rho_t - \rho_{xx} - ((\psi_x + h(t) \omega) \rho)_x = 0, \quad x \in \R,\ t > 
0, \\
\rho = \rho_0(x), \quad x \in \R
;
\quad
\rho_0(x) \ge 0, \ \int_{-\infty}^{\infty} \rho_0 (x) \cd x = 1,
\end{array}
\right.
\end{equation}
where $h$ is defined by \eqref{eq:switching-h} and for convenience $\sigma = 
1$.  We focus on the \emph{adiabatic limit} of~\eqref{eq:adiab-tilting}, i.e., 
on its behaviour when $T$, the period of the tilting, is large.  In this case 
we allow the diffusion to fully take its effect.  Thus the transport in the 
adiabatic limit can be said to be driven by diffusion.  Generally, the two 
tilted potentials $\psi(x) \pm \omega x$ corresponding to 
problem~\eqref{eq:adiab-tilting} are not symmetric and produce drift 
velocities of different absolute values.  For this reason the limiting average 
drift velocity of the tilting ratchet is nonzero.

The following theorem gives an effective formula for the adiabatic drift 
velocity.  Before we state it, we introduce some notations.  Consider the 
corresponding periodic problem
\begin{equation}
\label{eq:adiab-periodic}
g_t - g_{xx} - ((\psi_x + h(t) \omega) g)_x = 0 \quad x \in S^1, t > 0
; \quad \int_{S^1} g \cd x = 1
.
\end{equation}
It switches between the modes
\begin{equation} \label{eq:adiab-tilted-1}
g_t - g_{xx} - ((\psi_x + \omega) g)_x = 0 \quad x \in S^1
; \quad \int_{S^1} g \cd x = 1
\end{equation}
and
\begin{equation} \label{eq:adiab-tilted-2}
g_t - g_{xx} - ((\psi_x - \omega) g)_x = 0 \quad x \in S^1
; \quad \int_{S^1} g \cd x = 1
,
\end{equation}
spending a long time in each of them.

Let $g_+$ and $g_-$ be the stationary solutions of \eqref{eq:adiab-tilted-1} 
and \eqref{eq:adiab-tilted-2} respectively.  By $\gamma_+$ and $\gamma_-$ 
denote the inverses of the Log-Sobolev constants (see 
\eqref{eq:varlog-sobolev}) for the relative entropies
\begin{equation*}
\re_+[g] = \E[g|g_+], \ \re_-[g] = \E[g|g_-]
,
\end{equation*}
so by Corollary~\ref{cor:stable-attraction} for solutions 
of~\eqref{eq:adiab-tilted-1} we have the entropy decay
\begin{equation}
\label{eq:adiab-tilted-attr-1}
\re_+[g(\cdot, t)] \le \re_+[g(\cdot, t_0)] \e^{-\gamma_+ (t - t_0)} \quad (t 
\ge t_0)
,
\end{equation}
and similarly for solutions of~\eqref{eq:adiab-tilted-2} we have
\begin{equation}
\label{eq:adiab-tilted-attr-2}
\re_-[g(\cdot, t)] \le \re_-[g(\cdot, t_0)] \e^{-\gamma_- (t - t_0)} \quad (t 
\ge t_0)
.
\end{equation}

Let $g_\infty$ be the time-periodic solution of~\eqref{eq:adiab-periodic} with 
period $T$ and let $A(\omega) = A(\omega, \psi)$ and $A(-\omega) = A(-\omega, 
\psi)$ be defined according to \eqref{eq:defAB}.  As $g_\infty$ is 
time-periodic, the asymptotic drift velocity of \eqref{eq:adiab-tilting} can 
be expressed as
\begin{equation}
\label{eq:adiab-velocity}
v_{\infty}(T) =
- \frac 1T \int_0^T
\int_{0}^{1}
\psi_x(x) g_\infty(x, t)
\cd x
\cd t
\end{equation}
(see~\eqref{eq:tilting-asymptotic-velocity}).  Finally, put
\begin{equation}
\label{eq:adiab-oooo}
v_{\infty\infty} = - \frac {A(\omega) + A(-\omega)}2
= - \frac 12 \int_0^1 \psi_x(x) (g_+(x) + g_-(x)) \cd x
,
\end{equation}
where the last equality is due to \eqref{eq:A}.

\begin{theorem}
\label{th:adiabatic-limit}
Suppose $\psi \in C^4(\R)$ is $1$-periodic; then
\begin{equation}
\label{eq:adiabatic-limit}
|v_\infty(T) - v_{\infty\infty}|
\le
\frac {2^{3/2}}T
\max | \psi_x(x) |
\left(
\frac{\sqrt{\re_+[g_-]}}{\gamma_+}
+
\frac{\sqrt{\re_-[g_+]}}{\gamma_-}
\right)
+ o\left(\frac1T\right)
.
\end{equation}
In particular, if $v_{\infty\infty} \ne 0$, for large $T$ there is nonzero 
unidirectional transport.
\end{theorem}

Before we prove Theorem~\ref{th:adiabatic-limit} we must obtain some auxiliary 
results.  Put
\begin{equation*}
M_+ = \max \left | \ln \frac{g_-}{g_+} \right | \sqrt 2,
\quad
M_- = \max \left | \ln \frac{g_+}{g_-} \right | \sqrt 2.
\end{equation*}
and consider the functions
\begin{equation*}
\phi_+(R) = R + M_+ \sqrt R,
\quad
\phi_-(R) = R + M_- \sqrt R
.
\end{equation*}
Observe that $\phi_+$ and $\phi_-$ are continuous and increasing on $\R_+$.

\begin{lemma}
\label{lem:adiab-1}
For any probability density $g \in L^1(S^1)$ such that $g \ln g \in L^1(S^1)$, 
we have
\begin{gather}
| \re_+ [g] - \re_+ [g_-] | \le \phi_+(\re_- [g]),
\label{eq:adiab-lem11}
\\
| \re_- [g] - \re_- [g_+] | \le \phi_-(\re_+ [g])
.
\label{eq:adiab-lem12}
\end{gather}
\end{lemma}
\begin{proof}
A straightforward computation yields
\begin{multline*}
\re_+ [g] - \re_+ [g_-]
=
\int_{S^1}
g \ln \frac{g}{g_+}
\cd x
-
\int_{S^1}
g_- \ln \frac{g_-}{g_+}
\cd x
\\
=
\int_{S^1}
\left(
g \ln \frac{g}{g_-} + (g - g_-) \ln \frac{g_-}{g_+}
\right)
\dd x
,
\end{multline*}
whence
\begin{equation*}
| \re_+ [g] - \re_+ [g_-] |
\le
\re_-[g] + \frac{M_+}{\sqrt 2} \| g - g_- \|_{L^1(S^1)}
.
\end{equation*}
Now it remains to apply the Csiszár--Kullback inequality~\eqref{eq:cs-k} and 
obtain~\eqref{eq:adiab-lem11}.

Inequality~\eqref{eq:adiab-lem12} is proved by swapping $g_+$ and $g_-$.
\end{proof}

Note that the periodic solution $g_\infty$ of~\eqref{eq:adiab-periodic} 
implicitly depends on $T$, which is a parameter of tilting.
\begin{lemma}
\label{lem:adiab-2}
We have
\begin{equation}
\label{eq:adiab-lem2}
\lim_{T\to \infty} \re_+[g_\infty(\cdot, 0)] = \re_+[g_-]
;
\quad
\lim_{T\to \infty} \re_-[g_\infty(\cdot, T/2)] = \re_-[g_+]
.
\end{equation}
\end{lemma}
\begin{proof}
We only prove the first limit in~\eqref{eq:adiab-lem2}, as the proof of the 
second one is completely analogous.

Take an arbitrary $R>0$ and for any $\varepsilon > 0$ choose $T_\varepsilon > 
0$ in such a way that
\begin{gather*}
(\re_+ [g_-] + \varepsilon ) \e^{-\gamma_+T_\varepsilon/2} \le R,
\\
(\re_- [g_+] + \phi_-(R) ) \e^{-\gamma_- T_\varepsilon/2} \le 
\phi_+^{-1}(\varepsilon).
\end{gather*}
Now take $T > T_\varepsilon$ and let $\mathcal T$ be the resolving operator 
for~\eqref{eq:adiab-periodic} taking $g(\cdot, 0)$ to $g(\cdot, T)$.  We claim 
that $\mathcal T$ maps the set
\begin{equation*}
X = \{ g_0 \in W_2^1(S^1) \cap \{\text{probability densities}\}\mid | 
\re_+[g_0] - \re_+[g_-] | \le \varepsilon \}
\end{equation*}
into itself.  To prove this, consider a solution $g$ 
of~\eqref{eq:adiab-periodic} with the initial condition $g(\cdot, 0) = g_0 \in 
X$.  Thanks to the conservation of mass, it only remains to prove that
\begin{equation}
\label{eq:adiab-lem21}
| \re_+[g(\cdot, T)] - \re_+[g_-] | \le \varepsilon
.
\end{equation}
Indeed, making use of~\eqref{eq:adiab-lem11}, \eqref{eq:adiab-lem12}, and the 
attraction~\eqref{eq:adiab-tilted-attr-1} and~\eqref{eq:adiab-tilted-attr-2} 
we consequently obtain
\begin{gather*}
\re_+[g(\cdot, 0)] = \re_+ [g_0] \le \re_+ [g_-] + \varepsilon;
\\
\re_+[g(\cdot, T/2)]
\le \re_+[g_0] \e^{-\gamma_+ T/2}
\le ( \re_+ [g_-] + \varepsilon) \e^{-\gamma_+ T/2}
\le R;
\\
\re_-[g(\cdot, T/2)]
\le \re_-[g_+] + \phi_-(\re_+[g(\cdot, T/2)])
\le \re_-[g_+] + \phi_-(R);
\\
\re_-[g(\cdot, T)]
\le \re[g(\cdot, T/2)] \e^{-\gamma_- T/2}
\le (\re_-[g_+] + \phi_-(R)) \e^{-\gamma_- T/2}
\le \phi_+^{-1}(\varepsilon);
\\
| \re_+[g(\cdot, T)] - \re_+[g_-] |
\le \phi_+(\re_-[g(\cdot, T)])
\le \phi_+(\phi_+^{-1}(\varepsilon))
= \varepsilon
,
\end{gather*}
so~\eqref{eq:adiab-lem21} holds, and $X$ is invariant under $\mathcal T$.  
Moreover, $X$ is closed in $W_2^1(S^1)$, convex, bounded in $L^1(S^1)$, and by 
parabolic regularity $\mathcal T \colon X \to X$ is continuous and $\mathcal 
T(X)$ is precompact in $W_2^1(S^1)$.  By the Schauder fixed point theorem 
$\mathcal T$ has a fixed point in $X$, which is the initial data for a 
time-periodic solution of \eqref{eq:adiab-periodic}.  Due to uniqueness of 
such a periodic solution, this fixed point coincides with $g_\infty(\cdot, 
0)$.  This implies that $| \re_+[g_\infty(\cdot, 0)] - \re_+[g_-] | \le 
\varepsilon$ whenever $T \ge T_\varepsilon$, and the first limit 
in~\eqref{eq:adiab-lem2} is proved.
\end{proof}

\begin{proof}[Proof of Theorem~\ref{th:adiabatic-limit}]
Using \eqref{eq:adiab-velocity} and \eqref{eq:adiab-oooo}, we estimate the 
difference on the left-hand side of~\eqref{eq:adiabatic-limit} as follows:
\begin{multline*}
| v_\infty(T) - v_{\infty\infty} |
\\
=
\left|
- \frac 1T
\int_0^T
\int_0^1
\psi_x g_\infty
\cd x
\cd t
+
\frac 12
\int_0^1
\psi_x (g_+ + g_-)
\cd x
\right|
\\
=
\Bigg|
- \frac 1T
\int_0^{T/2}
\int_0^1
\psi_x g_\infty
\cd x
\cd t
- \frac 1T
\int_{T/2}^T
\int_0^1
\psi_x g_\infty
\cd x
\cd t
\\
+
\frac 1T
\int_0^{T/2}
\int_0^1
\psi_x g_+
\cd x
\cd t
+
\frac 1T
\int_{T/2}^T
\int_0^1
\psi_x g_-
\cd x
\cd t
\Bigg|
\\
\le
\frac 1T
\int_0^{T/2}
\int_0^1
|\psi_x (g_\infty - g_+) |
\cd x
\cd t
+
\frac 1T
\int_{T/2}^T
\int_0^1
|\psi_x (g_\infty - g_-) |
\cd x
\cd t
\\
\le
\frac 1T
\max_{x \in [0,1]} |\psi_x(x)|
\left(
\int_0^{T/2}
\| g_\infty - g_+ \|_{L^1(S^1)}
\cd t
+
\int_{T/2}^T
\| g_\infty - g_- \|_{L^1(S^1)}
\cd t
\right)
.
\end{multline*}
Applying the Csiszár--Kullback inequality, we obtain
\begin{multline*}
| v_\infty(T) - v_{\infty\infty} |
\\
\le
\frac {\sqrt 2}T
\max_{x \in [0,1]} |\psi_x(x)|
\left(
\int_0^{T/2}
\sqrt{\re_+[g_\infty]}
\cd t
+
\int_{T/2}^T
\sqrt{\re_-[g_\infty]}
\cd t
\right)
.
\end{multline*}
As $g_\infty(x, t)$ solves~\eqref{eq:adiab-tilted-1} for $t \in [0,T/2)$ and 
\eqref{eq:adiab-tilted-2} for $t \in [T/2,T)$, we can apply the entropy 
attraction \eqref{eq:adiab-tilted-attr-1} and \eqref{eq:adiab-tilted-attr-2} 
and obtain
\begin{multline*}
| v_\infty(T) - v_{\infty\infty} |
\\
\le
\frac {\sqrt 2}T
\max_{x \in [0,1]} |\psi_x(x)|
\bigg(
\sqrt{\re_+[g_\infty(\cdot, 0)]}
\int_0^{T/2}
\e^{-\gamma_+ t/2}
\cd t
\\
+
\sqrt{\re_-[g_\infty(\cdot, T/2)]}
\int_{T/2}^T
\e^{-\gamma_- (t-T/2)/2}
\cd t
\bigg)
\\
\le
\frac {2^{3/2}}T
\max_{x \in [0,1]} |\psi_x(x)|
\left(
\frac{1}{\gamma_+}
\sqrt{\re_+[g_\infty(\cdot, 0)]}
+
\frac{1}{\gamma_-}
\sqrt{\re_-[g_\infty(\cdot, T/2)]}
\right)
.
\end{multline*}
The last estimate and the limits~\eqref{eq:adiab-lem2} yield 
\eqref{eq:adiabatic-limit}.
\end{proof}
\begin{remark}
The proof of Lemma~\ref{lem:adiab-2} gives opportunity to estimate the term 
$o(1/T)$ on the right-hand side of \eqref{eq:adiabatic-limit}.
\end{remark}
\begin{remark}
If $\psi$ is fixed, $v_{\infty\infty} = - (A(\omega) + A(-\omega))/2$ is an 
analytic function of $\omega$.  Consequently, it either identically equals 0 
or has at most countably many zeroes without accumulation points.
\end{remark}
\begin{remark}
Having in mind~\eqref{eq:tilted-velocity}, we see that the adiabatic drift 
velocity $v_{\infty\infty}$ equals the arithmetic mean of the limiting drift 
velocities for the tilted potentials $\psi(x) \pm \omega x$.
\end{remark}
\begin{remark}
Observe that given $\psi$ and $\omega$, it is trivial to compute $A(\omega)$ 
numerically.  In this sense Theorem~\ref{th:adiabatic-limit} is effective.
\end{remark}

\subsection{Bulk transport direction}

There is another formula that allows one to determine the direction of the 
adiabatic transport.  Put
\begin{equation}
\label{eq:defJ}
%\textstyle
J = J(\psi, \omega)
=
\frac{2\displaystyle \int_0^1
\int_0^x
\sinh (\psi(x) - \psi(y)) \sinh [\omega (x - y - 1/2)]
\cd y
\cd x}{\sinh(\omega/2)}
.
\end{equation}

\begin{prop}
\label{pr:shsh}
If $\omega \ne 0$, the sign of $v_{\infty\infty}$ coincides with the sign of 
$J$.  Consequently, if $J > 0$ ($J < 0$), then the adiabatic transport goes in 
the positive (respectively, negative) direction.
\end{prop}
\begin{proof}
Using the definition of $A$ \eqref{eq:defAB}, write
\begin{multline*}
\frac{1}{A(\omega)}
=
\frac{1}{\e^\omega - 1}
\Bigg(
(\e^\omega - 1)
\int_0^1
\int_0^x
\e^{\psi(y) - \psi(x) + \omega(y - x)}
\cd y
\cd x
\\
+
\int_0^1
\int_0^1
\e^{\psi(y) - \psi(x) + \omega(y - x)}
\cd y
\cd x
\Bigg)
\\
=
\frac{1}{\e^\omega - 1}
\Bigg(
\e^\omega
\int_0^1
\int_0^x
\e^{\psi(y) - \psi(x) + \omega(y - x)}
\cd y
\cd x
\\
+
\int_0^1
\int_x^1
\e^{\psi(y) - \psi(x) + \omega(y - x)}
\cd y
\cd x
\Bigg)
\\
=
\frac{\e^{\omega/2}}{\e^{\omega} - 1}
\int_0^1
\int_0^x
\Big(
\e^{\psi(y) - \psi(x) + \omega(y - x + 1/2)}
\\
+
\e^{-(\psi(y) - \psi(x) + \omega(y - x + 1/2))}
\cd y
\cd x
\Big)
\\
=
\frac{1}{\sinh (\omega/2)}
\int_0^1
\int_0^x
\cosh (\psi(y) - \psi(x) + \omega(y-x+1/2))
\cd y
\cd x
.
\end{multline*}
Substituting $- \omega$ for $\omega$, we obtain
\begin{equation*}
\frac{1}{A(-\omega)}
=
-
\frac{\int_0^1
\int_0^x
\cosh (\psi(y) - \psi(x) - \omega(y-x+1/2))
\cd y
\cd x}{\sinh (\omega/2)}
.
\end{equation*}
Summing and converting the difference of hyperbolic cosines into product, we 
get
\begin{equation*}
\frac{1}{A(\omega)}
+
\frac{1}{A(-\omega)}
=
J
.
\end{equation*}
Now it suffices to observe that as $A(\omega)$ and $A(-\omega)$ have opposite 
signs, so do the sums $A(\omega) + A(-\omega) = - 2 v_{\infty\infty}$ and 
$1/A(\omega) + 1/A(-\omega) = J$.
\end{proof}

\begin{example}
\emph{Symmetric potentials} satisfying $\psi(x) = \psi(1 - x)$ do not produce 
adiabatic transport.  This follows e.g.\ from Proposition~\ref{pr:shsh}.  
Indeed, if $\psi$ is symmetric, by changing the variables $x' = 1 - y$, $y' = 
1 - x$ in~\eqref{eq:defJ} we get $J(\psi, \omega) = - J(\psi, \omega)$, whence 
$J = 0$.  Note however, that symmetric potentials can produce transport if the 
tilting regime is asymmetric in time unlike~\eqref{eq:switching-h}, see below.
\end{example}
\begin{example}
\emph{Supersymmetric potentials} (see~\cite{rei02})  satisfying $- \psi(x) = 
\psi(x + 1/2)$ do not produce adiabatic transport either.  This, too, can be 
derived from Proposition~\ref{pr:shsh}.  Indeed, utilizing in~\eqref{eq:defJ} 
the change of variables $x' = y + 1/2$, $y' =x - 1/2$ on the set 
$$Q=\left\{(x,y):\frac 1 2 \leq x \leq 1,\, 0\leq y\leq \frac 1 2 \right\},$$ 
and  the change of variables $x'' = x + 1/2$, $y'' = y + 1/2$ on the rest of 
the triangle $$\left\{(x,y):\, 0\leq x \leq 1,\, 0\leq y\leq  x 
\right\}\setminus Q,$$
we deduce that $J = 0$.
\end{example}
\begin{example}
Simple examples of asymmetric potentials such as $\psi(x) = \cos (2\pi x^m)$ 
suggest that if $\psi$ increases (decreases) on a major interval, then the 
direction of adiabatic transport is positive (resp.\ negative).  The following 
theorem justifies this claim.
\end{example}

\begin{theorem}
\label{th:adiab-foo}
Suppose that $\psi \in C^4(S^1)$ strictly increases along an oriented arc 
$\overrightarrow{[\alpha,\beta]}$.
Let $h \colon S^1 \times [0,1] \to S^1$ be a homotopy such that
\begin{enumerate}
\item
$h(\cdot, 0)$ is the identity mapping on $S^1$;
\item
for any $\lambda \in [0,1)$ the mapping $h(\cdot, \lambda) \colon S^1 \to S^1$ 
is $C^4$;
\item
$h(\cdot, 1)$ preserves the orientation on the oriented arc 
$\overrightarrow{[\alpha,\beta]}$;
\item
$h(\cdot, 1)$ maps the oriented arc $\overrightarrow{[\beta,\alpha]}$ onto a 
single point.
\end{enumerate}
Let $v_{\infty\infty}(\lambda)$ be the asymptotic velocity corresponding to 
$\psi_\lambda := \psi(h(\cdot, \lambda))$, where $0 \le \lambda < 1$.  Then 
$v_{\infty\infty}(\lambda) > 0$ if $\lambda$ is sufficiently close to 1.
\end{theorem}
\begin{proof}
Without loss of generality, $\omega > 0$ and $\alpha = 0 = h(\alpha,\lambda)$ 
for $\lambda \in [0,1]$.  We use Proposition~\ref{pr:shsh}.  Since for any $x 
\in S^1$ we have $\psi_\lambda(x) \to \psi_1(x)$ as $\lambda \to 1$, by the 
Lebesgue dominated convergence theorem we obtain
\begin{multline}
\label{eq:adiab-foo}
J(\omega, \psi_\lambda) \to J(\omega, \psi_1)
\\
=
\frac{2}{\sinh(\omega/2)}
\int_0^1
\int_0^x
\sinh (\psi_1(x) - \psi_1(y)) \sinh \omega (x - y - 1/2)
\cd y
\cd x
.
\end{multline}
The homotopy $h$ maps the interior of the arc $\overrightarrow{[\alpha, 
\beta]}$ onto $S^1 \setminus \{0\}$ and preserves the orientation, so $\psi_1$ 
increases on the interval $(0, 1)$.  Now the right-hand side 
of~\eqref{eq:adiab-foo} is positive by
Proposition~\ref{pr:phiphi} (see below).  Consequently, if $\lambda$ is 
sufficiently close to 1, $J(\omega, \psi_\lambda)$ is also positive and so is 
$v_{\infty\infty}(\lambda)$.
\end{proof}

\begin{remark}
In the notations of Theorem~\ref{th:adiab-foo} the exact formula for the 
limiting velocity is
\begin{equation}
\label{eq:vinftyinfty}
\lim_{\lambda\to 1} v_{\infty\infty}(\lambda)
=
-
\frac{A(\omega, \psi_1) + A(-\omega, \psi_1)}2
.
\end{equation}
It does not depend on the values of the initial potential $\psi$ on the arc 
$\overrightarrow{[\beta, \alpha]}$ that collapses to a point.
\end{remark}

It remains to prove the following functional inequality.
\begin{prop}
\label{pr:phiphi}
Let $f:[0,1]\to \R$ be any increasing continuous function, 
$\varphi:[0,+\infty)\to \R$ be a convex $C^1$-smooth function, and 
$\Phi:\R\to\R$ be an odd continuous function such that $x>0$ implies 
$\Phi(x)>0$.  Then
\begin{equation}
\label{sinineq} \int\limits_0^1\int\limits_0^x 
\varphi\left(f(x)-f(y)\right)\Phi\left(x-y-1/2\right)\,dy\,dx \geq 0.
\end{equation}
Moreover, if $f$ is strictly increasing and $\varphi$ is strictly convex (in 
the sense that $\varphi'$ is strictly increasing), then inequality 
\eqref{sinineq} is strict.
\end{prop}

We first prove two auxiliary statements.
\begin{lemma} \label{Lemm1d} Let $\varphi$ be as above, and 
$x_1,\dots,x_{2m}$, $m\in \mathbb{N}$, be a collection of non-negative 
numbers.  Then
\begin{equation}\label{sin1d} \sum\limits_{i=1}^{m}\varphi\left(\sum\limits_{j=1}^{m+1}x_{i+j-1}\right)\geq \sum\limits_{i=1}^{m+1}\varphi\left(\sum\limits_{j=1}^{m}x_{i+j-1}\right).\end{equation} 
If $\varphi$ is strictly convex and the numbers $x_1,\dots,x_{2m}$ are positive, then this inequality is strict.
\end{lemma}

\begin{proof}
It suffices to observe that all the partial derivatives of the function 
$F(x_1,\dots,x_{2m})=\sum\limits_{i=1}^{m}\varphi\left(\sum\limits_{j=1}^{m+1}x_{i+j-1}\right)- 
\sum\limits_{i=1}^{m+1}\varphi\left(\sum\limits_{j=1}^{m}x_{i+j-1}\right)$ are 
non-negative (and even positive provided $\varphi$ is strictly convex and  
$x_1,\dots,x_{2m}$ are positive), and $F(0,\dots,0)=0$.
\end{proof}
\begin{lemma}
\label{Lemm1c}
Let $\varphi$ and $f$ be as above, and $a<b$ be two positive numbers.  Then
\begin{equation}\label{sin1c} \int\limits_0^a \varphi\left(f(x+b)-f(x)\right)\, dx\geq \int\limits_0^b \varphi\left(f(x+a)-f(x)\right)\,dx.\end{equation} 
If $f$ is strictly increasing, $\varphi$ is strictly convex, and $a/b$ is a rational number, then inequality \eqref{sin1c} is strict.
\end{lemma}

\begin{proof}
Assume first that the ratio $a/b$ is equal to $m/m+1$ with some natural number 
$m$.  After rescaling, without loss of generality we may suppose that $a=m$. 
In this case, inequality \eqref{sin1c} may be rewritten as
\begin{multline} 
\sum\limits_{i=1}^{m} \int\limits_0^1\varphi\left( f(x+i+m)-f(x+i-1)\right)\,dx\\ \geq \sum\limits_{i=1}^{m+1} \int\limits_0^1\varphi\left(f(x+i+m-1)-f(x+i-1)\right)\,dx, \end{multline} or
\begin{multline} 
\sum\limits_{i=1}^{m} \int\limits_0^1\varphi\left(\sum\limits_{j=1}^{m+1} f(x+i+j-1)-f(x+i+j-2)\right)\,dx\\ \geq \sum\limits_{i=1}^{m+1} \int\limits_0^1\varphi\left(\sum\limits_{j=1}^{m}f(x+i+j-1)-f(x+i+j-2)\right)\,dx, \end{multline} and thus the statement of the lemma follows from Lemma \ref{Lemm1d}. 

If $a/b$ is a rational number $\frac mn$, with $m,n\in\mathbb{N}$, $m<n$, then 
it can be decomposed  as $\frac m{m+1} \cdot \frac {m+1}{m+2} \cdot{}\cdots 
{}\cdot \frac {n-1}{n}$, and the statement of the lemma follows from the 
previous case.  By continuity, the non-strict inequality \eqref{sin1c} holds 
for all irrational $a/b$.
\end{proof}

\begin{proof}[Proof of Proposition~\ref{pr:phiphi}]
We make a change of variables in the double integral \eqref{sinineq}, and 
conclude that it is equal to
\begin{multline}
\label{sinphi}
\int\limits_{-1/2}^{1/2}\int\limits_0^{1/2-\xi} 
\varphi\left(f(y+\xi+1/2)-f(y)\right)\Phi(\xi)\,dy\,d\xi \\ 
=\int\limits_{0}^{1/2}\Big(\int\limits_0^{1/2-\xi} 
\varphi\left(f(y+\xi+1/2)-f(y)\right)\,dy\\-\int\limits_0^{1/2+\xi} 
\varphi\left(f(y-\xi+1/2)-f(y)\right)\,dy\Big)\Phi(\xi)\,d\xi \geq 
0
\end{multline}
by Lemma \ref{Lemm1c}.  When $f$ is strictly increasing and $\varphi$ is 
strictly convex, this integral is strictly positive since the set 
\begin{multline}\Big\{\xi\in\left(0,\frac 12\right)\Big| 
\int\limits_0^{1/2-\xi} \varphi\left(f(y+\xi+1/2)-f(y)\right)\,dy\\ 
>\int\limits_0^{1/2+\xi} 
\varphi\left(f(y-\xi+1/2)-f(y)\right)\,dy\Big\}\end{multline} is open and 
non-empty (since it contains the rational numbers), and therefore has non-zero 
Lebesgue measure.
\end{proof}

\section{Semiadiabatic limit and the stochastic Stokes' drift} \label{slssd}

Consider the tilting ratchet with the potential
\begin{equation}
\label{eq:semiad-tilt}
\Psi(x, t) = \psi(x) + H(t)x
,
\end{equation}
where
\begin{equation}
\label{eq:semiad-added-tilt}
H(t) =
\begin{cases}
\Omega , & 0 < t< \tau, \\
\omega , & \tau < t < T
,
\end{cases}
\end{equation}
and as usual $\psi(x)$ is $1$-periodic in $x$ and $H(t)$ is $T$-periodic in 
$T$.  We assume that the non-bias condition~\eqref{eq:unbiasedH} is satisfied, 
i.e.,
\begin{equation}
\label{eq:semiad-cond-unbiased}
\Omega \tau + \omega(T - \tau) = 0
.
\end{equation}
We regard $T$, $\tau$, and $\omega$ as the independent parameters of the 
tilting ratchet; $\Omega$ can be expressed via the independent parameters by 
means of~\eqref{eq:semiad-cond-unbiased}.

We study the regime~\eqref{eq:semiad-tilt} in the so-called \emph{semiadiabatic 
limit} supposing that $T$ is large, $\tau/T$ is small, and $\omega$ is 
constant.  It follows from~\eqref{eq:semiad-cond-unbiased} that in the 
semiadiabatic limit $|\Omega|$ is large.

As before, let $v_\infty = v_\infty(\omega, T, \tau)$ denote the eventual 
drift velocity given by~\eqref{eq:def-asymptotic-velocity}.  Further, let 
$g^*$ be the $1$-periodic solution of the problem
\begin{equation}
\label{eq:semiad-defg}
g_{xx} + ((\psi_x + \omega)g)_x = 0, \quad \int_0^1 g(x) \cd x = 1
,
\end{equation}
which also satisfies
\begin{equation*}
g^*_x + (\psi_x + \omega)g^* = A(\omega)
\end{equation*}
with $A(\omega)$ defined by~\eqref{eq:defAB} (see 
Remark~\ref{rem:tilted-velocity}).  Put
\begin{equation}
\label{eq:semiad-limit-velocity}
v_{\infty\infty}(\omega)
=
-
\int_0^1
\psi_x g^*
\cd x
\equiv
\omega - A(\omega)
.
\end{equation}
Our main result concerning the semiadiabatic tilting is that 
$v_{\infty\infty}(\omega)$ is the asymptotic average drift velocity as $T \to 
\infty$ and $\tau/T \to 0$, and its sign coincides with the sign of $\omega$ 
provided that $\psi$ is nontrivial.  For definitiveness we assume that $\omega > 0$.

First we address the positivity of $v_{\infty\infty}$.  It will follow from 
the next proposition.

\begin{prop}
\label{pr:funct}
Given $\omega > 0$, the functional
\begin{equation}
\label{funct} 
\mathcal{V}_\omega(F)=\alpha(\omega)\int\limits_0^1\int\limits_0^x \frac 
{F(y)}{F(x)}\cd y\cd x+\int\limits_0^1\int\limits_0^1 \frac {F(y)}{F(x)}\cd 
y\cd x,
\end{equation}
where $F\in C[0,1],\ F(x)>0$ for all $x\in[0,1]$, attains its global minimum 
$\alpha(\omega)/\omega$ when and only when $F(x)=C\e^{\omega x}$.
\end{prop}

\begin{proof}
Since the functional is homogeneous of degree zero, it suffices to prove that 
under the additional restriction
\begin{equation}
\int\limits_0^1 F(y)\cd y=1,
\end{equation}
the only minimizer is $F(x)=\frac{\omega}{\alpha(\omega)} \e^{\omega x}$. 

Let $G(x)=1+\alpha(\omega)\int\limits_0^x F(y)\cd y$.  Then $G(0)=1,$ 
$G(1)=\e^\omega$, $G'(\cdot)>0$, and our functional becomes
\begin{equation}
\label{funct1}
\tilde{\mathcal{V}}_\omega(G) =\alpha(\omega)\int\limits_0^1\frac 
{G(x)}{G'(x)}\cd x
.
\end{equation} 

Let now $H(x)=\ln G(x)$.  Then $H(0)=0,$ $H(1)=\omega,$ $H'(\cdot)>0$, and 
\begin{equation}
\label{funct2}
\tilde{\tilde{\mathcal{V}}}_\omega(H) =\alpha(\omega)\int\limits_0^1\frac 
1{H'(x)}\cd x
.
\end{equation}

The Cauchy--Bunyakovskii--Schwarz inequality implies that \begin{equation}
\int\limits_0^1\frac 1{H'(x)}\cd x\geq \frac 1 {\int\limits_0^1 H'(x)\cd x}= 
\frac 1 {\omega}
,
\end{equation}
and the equality holds only if $H'(x)\equiv \omega$.  This means that the 
minimum of $\mathcal{V}_\omega$ is achieved merely if $H(x)=\omega x$, 
$G(x)=\e^{\omega x}$ and $F(x)=\frac{\omega}{\alpha(\omega)} \e^{\omega x}$. 
\end{proof}

\begin{remark}
It is interesting to observe that the functionals \eqref{funct} and 
\eqref{funct1} neither are convex nor fit into the framework of general 
$L^1$-lower semicontinuity criteria which go back to \cite{ser61} (see e.g.\ 
\cite{gori} for a review).
\end{remark}

\begin{prop}
\label{pr:semiad-bounds}
Given a $1$-periodic $\psi \in C^4 (\R)$ and $\omega > 0$, $v_{\infty\infty}$ 
satisfies
\begin{equation}
\label{eq:semiad-bounds}
0 \le v_{\infty\infty} \equiv \omega - A(\omega) < \omega
.
\end{equation}
Moreover, $v_{\infty\infty} = 0$ if and only if $\psi$ is constant.
\end{prop}
\begin{proof}
Elementary properties of $A$ (see Remark~\ref{rem:A}) yield the upper bound 
in~\eqref{eq:semiad-bounds}.  The lower bound is a corollary of 
Proposition~\ref{pr:funct}.  Indeed, put $F(x) = \e^{\psi(x) + \omega x}$ and 
write
\begin{equation*}
A = \frac{\alpha(\omega)}{\mathcal V_\omega(F)}
,
\end{equation*}
where $\mathcal V_\omega$ is given by~\eqref{funct}.  According to 
Proposition~\ref{pr:funct}, we have
\begin{equation*}
A \le \frac{\alpha(\omega)}{\alpha(\omega)/\omega} = \omega,
\end{equation*}
and the lower bound in~\eqref{eq:semiad-bounds} is proved.  Moreover, the 
bound is attained if and only if $F(x) = C\e^{\omega x}$, with some constant 
$C > 0$, i.e., if and only if $\psi(x) \equiv \ln C$ is constant.
\end{proof}

Write the Log-Sobolev inequality associated with $g^*$ in the form
\begin{equation}
\label{eq:semiad-log}
\int_{S^1}
g \ln \frac{g}{g^*}
\cd x
\le
\frac 1\gamma
\int_{S^1}
g \left|\left(\ln \frac{g}{g^*}\right)_x \right|^2
\cd x
.
\end{equation}

Put
\begin{equation}
\label{eq:defM}
M_0 = \max_{x \in [0,1]} | \psi_x(x) |
,
\ M_1 = \max_{x \in [0,1]} \left| \frac{g^*_x(x)}{g^*(x)} \right|
.
\end{equation}

The following statements characterize the semiadiabatic limit of the tilting 
ratchets.
\begin{theorem}
\label{th:semiad-velocity}
Suppose that $\psi \in C^4(\R)$ is $1$-periodic and $\omega > 0$, $0 < \tau < 
T$; then the average drift velocity $v_{\infty}(\omega, T, \tau)$ satisfies
\begin{multline}
\label{eq:semiad-velocity}
|v_{\infty}(\omega, T, \tau) - v_{\infty\infty}(\omega)|
\le
(M_0 + v_{\infty\infty}(\omega))
\frac \tau T
\\
+
\frac{
2^{3/2} M_0 M_1^{1/2}
}{
\gamma
}
\omega^{1/2}
\frac{1}{T^{1/2}}
\left(
1 +
\frac{
1
}{
\e^{\gamma T (1 - \tau/T)} - 1
}
\right)
^{1/2}
.
\end{multline}
\end{theorem}

The proof will be given later in this section.

\begin{cor}
\label{cor:semiad-lim}
Under the hypothesis of Theorem~\ref{th:semiad-velocity}, for any $\omega > 
0$, we have
\begin{equation}
\lim_{
\substack{
T \to \infty \\
\tau/T \to 0
}}
v_{\infty}(\omega, T, \tau)
=
v_{\infty\infty}(\omega)
,
\end{equation}
the limit is uniform in $\omega \in (0, \omega_0]$ for any $\omega_0 > 0$, and 
the rate of convergence is $O(T^{-1/2} + \tau/T)$.
\end{cor}
\begin{proof}
It suffices to observe that by Proposition~\ref{pr:semiad-bounds} we have 
$v_{\infty\infty}(\omega) \le \omega$ and that
\begin{equation}
1 +
\frac{
1
}{
\e^{\gamma T (1 - \tau/T)} - 1
}
\end{equation}
remains bounded as $T \to \infty$ and $\tau/T \to 0$.
\end{proof}
Combining Corollary~\ref{cor:semiad-lim} and 
Proposition~\ref{pr:semiad-bounds}, we get the following corollary.
\begin{cor}
\label{cor:semiad-pos}
For any non-constant $1$-periodic $\psi \in C^4$ and any $\omega > 0$, we have 
$v_{\infty\infty}(\omega) > 0$, if $T$ is sufficiently large and $\tau/T$ is 
sufficiently small.
\end{cor}
In other words, Corollary~\ref{cor:semiad-pos} means that \emph{the 
semiadiabatic transport goes in the positive/negative direction according to 
the sign of $\omega$}.

Before proving Theorem~\ref{th:semiad-velocity}, we introduce some notations 
and prove a lemma.  Denote by
\begin{equation*}
\re[g] = \E[g | g^*]
\end{equation*}
the relative entropy with respect to $g^*$, and let
\begin{equation*}
\ep[g] =
\int_{S^1}
g \left|\left(\ln \frac{g}{g^*}\right)_x \right|^2
\cd x
\end{equation*}
be the corresponding entropy production term, then the Log-Sobolev 
inequality~\eqref{eq:semiad-log} can be written as
\begin{equation}
\label{eq:semiad-varlog}
\re[g] \le \frac 1 \gamma \ep[g]
.
\end{equation}

In the case of semiadiabatic tilting, equation~\eqref{eq:per-pr} splits into
\begin{gather}
g_t - g_{xx} - ((\psi_x + \Omega) g)_x = 0,
\quad (t, x) \in (0, \tau) \times S^1;
\label{eq:semiad-per1}
\\
g_t - g_{xx} - ((\psi_x + \omega) g)_x = 0,
\quad (t, x) \in (\tau, T) \times S^1
.
\label{eq:semiad-per2}
\end{gather}
We consider \eqref{eq:semiad-per1}--\eqref{eq:semiad-per2} in the class
\begin{equation}
\label{eq:semiad-pdf}
g \ge 0; \quad
\int_{S^1} g(x, t) \cd x = 1
.
\end{equation}
As usual, $g_\infty$ denotes the unique $T$-periodic in $t$ solution of 
\eqref{eq:semiad-per1}--\eqref{eq:semiad-pdf}, existing due to 
Theorem~\ref{th:biper-sol}.  Note that the function $g_\infty$ itself 
implicitly depends on the parameters of tilting $T$, $\tau$, and $\omega$.
\begin{lemma}
\label{lem:semiad-l1}
We have
\begin{equation}
\label{eq:semiad-l1}
\re[g_\infty(\cdot, \tau)]
\le
M_1 \omega T
\left(
1+
\frac{1}{
\e^{\gamma T(1-\tau/T)} - 1
}
\right)
.
\end{equation}
\end{lemma}
\begin{proof}
We start with an \emph{a priori} estimate for a solution $g$ 
of~\eqref{eq:semiad-per1}--\eqref{eq:semiad-pdf}.  First we estimate the 
entropy for $t \in [0, \tau]$.  Put $r = g/g^*$, then $r_t = g_t/g^*$, and 
$(\ln r)_t = g_t/g$.  We have:
\begin{equation*}
\frac{\dd}{\dd t} \re[g]
=
\frac{\dd}{\dd t}
\int_{S^1}
g \ln r
\cd x
=
\int_{S^1}
g_t \ln r
\cd x
+
\int_{S^1}
g_t
\cd x
.
\end{equation*}
Plugging in~\eqref{eq:semiad-per1}, using the conservation of 
mass~\eqref{eq:semiad-pdf} and integrating by parts, we proceed as follows:
\begin{multline*}
\frac{\dd}{\dd t} \re[g]
=
\int_{S^1}
(g_{xx} + ((\psi_x + \Omega) g)_x) \ln r
\cd x
\\
=
-
\int_{S^1}
g
\frac{r_x^2}{r^2}
\cd x
+
\int_{S^1}
\frac{g}{r}
(r_{xx} - (\psi_x + \Omega) r_x)
\cd x
\\
=
-
\ep[g]
+
\int_{S^1}
g^*
(r_{xx} - (\psi_x + \Omega) r_x)
\cd x
\\
\le
\int_{S^1}
r
(g^*_{xx} + ((\psi_x + \Omega)g^*)_x)
\cd x
.
\end{multline*}
As $g^*$ solves~\eqref{eq:semiad-defg}, we obviously have
\begin{equation*}
g^*_{xx} + ((\psi_x + \Omega)g^*)_x = (\Omega - \omega) g^*_x,
\end{equation*}
whence
\begin{equation*}
\frac{\dd}{\dd t} \re[g]
\le
(\Omega - \omega)
\int_{S^1}
r
g^*_x
\cd x
=
(\Omega - \omega)
\int_{S^1}
g
\frac{g^*_x}{g^*}
\cd x
,
\end{equation*}
and using the conservation of mass once again, we obtain
\begin{equation*}
\frac{\dd}{\dd t} \re[g]
\le
M_1 |\Omega - \omega|
.
\end{equation*}
This yields
\begin{equation*}
\re[g(\cdot, \tau)]
\le
M_1 \tau |\Omega - \omega|
+
\re[g(\cdot, 0)]
\end{equation*}
whenever $\re[g(\cdot, 0)]$ exists.  Using~\eqref{eq:semiad-cond-unbiased}, we 
can write the last inequality in the form
\begin{equation}
\label{eq:semiad-l11}
\re[g(\cdot, \tau)]
\le
M_1 \omega T
+
\re[g(\cdot, 0)]
.
\end{equation}

Now for $t \in [\tau, T]$ the function $g$ solves 
equation~\eqref{eq:semiad-per2} and $g^*$ is the stationary solution of the 
same equation, so by Corollary~\ref{cor:stable-attraction}
\begin{equation*}
\re[g] \le \re[g(\cdot, \tau)] \e^{-\gamma (t - \tau)}
.
\end{equation*}
From this inequality and~\eqref{eq:semiad-l11} we get
\begin{equation}
\label{eq:semiad-l12}
\re[g(\cdot, T)] \le (M_1 \omega T + \re[g(\cdot, 0)] )\e^{-\gamma (T - \tau)}
.
\end{equation}

Denote by $\mathcal T$ the operator taking a probability density $u \in 
L^1(S^1)$ to $g(\cdot, T)$, where $g$ 
solves~\eqref{eq:semiad-per1}--\eqref{eq:semiad-per2} with the initial 
condition $g|_{t = 0} = u$, and put
\begin{equation}
R
=
\frac{
M_1 \omega T
}{
\e^{\gamma (T-\tau)} - 1
}
.
\end{equation}
It follows from~\eqref{eq:semiad-l12} that $\mathcal T$ maps the set
\begin{equation*}
X =
\{
u \in W_2^1(S^1) \cap \{\text{probability densities}\}
\mid
\re[u] \le R
\}
\end{equation*}
into itself.  Moreover, $X$ is a convex closed subset of $W_2^1(S^1)$ and by 
parabolic regularity $\mathcal T \colon X \to X$ is continuous and $\mathcal 
T(X)$ is precompact in $W_2^1(S^1)$.  Hence by the Schauder fixed point 
theorem $\mathcal T$ has a fixed point in $X$, which is the initial data for a 
time-periodic solution of \eqref{eq:semiad-per1}--\eqref{eq:semiad-per2}.  
However, such periodic solution is unique, so the fixed point found above 
coincides with $g_\infty(\cdot, 0)$.  Therefore, $\re [g_\infty(\cdot, 0)] \le 
R$ and \eqref{eq:semiad-l1} follows from \eqref{eq:semiad-l11}.
\end{proof}
\begin{proof}[Proof of Theorem~\ref{th:semiad-velocity}.]
We have:
\begin{multline}
\label{eq:semiad-velocity-1}
| v_{\infty}(\omega, T, \tau) - v_{\infty\infty}(\omega) |
\\
=
\left|
-
\frac 1T
\int_0^T
\int_{S^1}
\psi_x g_\infty
\cd x
\cd t
+
\int_{S^1}
\psi_x
g^*
\cd x
\right|
\\
\le
\frac 1T
\left|
\int_0^\tau
\int_{S^1}
\psi_x g_\infty
\cd x
\cd t
\right|
\\
+
\frac 1T
\left|
\int_\tau^T
\int_{S^1}
\psi_x (g_\infty - g^*)
\cd x
\cd t
\right|
+
\frac \tau T
\left|
\int_{S^1}
\psi_x
g^*
\cd x
\right|
\\
\le
\frac \tau T
(M_0 + v_{\infty\infty}(\omega))
+
\frac{M_0}{T}
\int_\tau^T \| g_\infty - g^* \|_{L^1(S^1)}
\cd t
\end{multline}
(here we have used the conservation of mass for $g_\infty$).  We now estimate 
the last term.  Using the Csiszár--Kullback inequality and the 
attraction~\eqref{eq:stable-attraction}, we have
\begin{multline*}
\frac{M_0}{T}
\int_\tau^T \| g_\infty - g^* \|_{L^1(S^1)}
\cd t
\le
\frac{M_0}{T}
\int_\tau^T \sqrt{2 \re [g_\infty]}
\cd t
\\
\le
\frac{M_0}{T}
\sqrt{2 \re [g_\infty(\cdot, \tau)]}
\int_\tau^T \e^{- \gamma (t - \tau)/2}
\cd t
\le
\frac 1T
\frac{2M_0}{\gamma}
\sqrt{2 \re [g_\infty(\cdot, \tau)]}
.
\end{multline*}
Estimating the entropy~$\re [g_\infty(\cdot, \tau)]$ by means 
of~\eqref{eq:semiad-l1}, we have
\begin{equation*}
\frac{M_0}{T}
\int_\tau^T \| g_\infty - g_* \|_{L^1(S^1)}
\cd t
\le
\frac 1T
\frac{2^{3/2}M_0}{\gamma}
\left(
M_1 \omega T
\left(
1+
\frac{1}{
\e^{\gamma T(1-\tau/T)} - 1
}
\right)
\right)^{\frac 12}
.
\end{equation*}
Combining this inequality with~\eqref{eq:semiad-velocity-1}, we 
obtain~\eqref{eq:semiad-velocity}.
\end{proof}
%\section {Stokes' drift and minimization}

%\section {The flashing ratchet revisited}

We finish this section by considering the travelling potential  
\begin{equation}
\label{eq:trav}
\Psi(x, t) = \psi(x-\omega t)
,
\end{equation}
where $\psi$ is $1$-periodic and $\omega$ is a constant.  The corresponding 
ratchet model  \cite{bdk08,bdk09} coincides with the one for the stochastic 
Stokes' drift \cite{bena00,jl98}.  In \cite{bdk08} it was conjectured that the 
average drift velocity of this ratchet is positive for positive $\omega$ 
(cf.~also \cite[Section 4.4.1]{rei02}).  It is straightforward to check (cf. 
\cite{bdk09}) that the corresponding periodic solution of \eqref{eq:per-pr} 
with $F = - \Psi_x$ satisfying~\eqref{eq:per-mass-conservation} is
$$g_\infty=g^*(x-\omega t),$$
where $g^*$ is defined above in this section.  Then the corresponding bulk 
velocity is
\begin{multline}
\label{eq:trav-limit-velocity}
v_{\infty}
=
-
\frac 1T\int_0^T\int_0^1
\psi_x (x-\omega t)g^*(x-\omega t)
\cd x \cd t
\\= -
\int_0^1
\psi_x (x)g^*(x)
\cd x \equiv
\omega - A(\omega)
.
\end{multline} The conjecture of \cite{bdk08} (for any non-constant $\psi$) 
thus follows from our Proposition \ref{pr:semiad-bounds}.

\section {Small diffusion coefficient} \label{sdc}

We are now interested in applying the bulk velocity representation \eqref{eq:def-asymptotic-velocity} for finding sufficient conditions of transport in the case of small~$\sigma$. 

The result concerns generic piecewise smooth potentials $\Psi(x,t)$ which are $T$-periodic in $t$ and whose $x$-derivatives are $1$-periodic in~$x$. 

\begin{theorem} \label{odeth} Assume that the ODE  \begin{equation}\label{ode1} y'(t)=F(y(t),t) 
\end{equation} (where $F=-\Psi_x$ is as in Theorem \ref{th:biper-sol}) 
has no $T$-periodic solutions.  Then $v_\infty\neq 0$ for sufficiently small 
$\sigma$.  Moreover,
\begin{equation}
\label{eq:ineq1} 
\operatorname{sign}v_\infty=\mathrm{sign}(y(T)-y(0))
\end{equation}
for every solution $y$ to \eqref{ode1}.
\end{theorem}
\begin{proof}
Given a sequence $\sigma_n\to 0$, consider the corresponding $T$-periodic 
solutions $g_{\infty n}$ to~\eqref{eq:per-pr} 
satisfying~\eqref{eq:per-mass-conservation}, and the bulk velocities 
$v_{\infty n}$ defined by~\eqref{eq:def-asymptotic-velocity}.  If the theorem 
is not true, $\{\sigma_n\}$ can be chosen in such a way that every $v_{\infty 
n}$ violates~\eqref{eq:ineq1}.  Consider the following  auxiliary equation
\begin{equation}
\label{eq:per-pr0}
\zeta_t + (F \zeta)_x = 0, \quad (t, x) \in (0, \infty) \times \R.\end{equation}
It is easy to see that the solution $\zeta$ of \eqref{eq:per-pr0} can be 
expressed in the form $\zeta(x,t)=[z(s(x,t))]_x$, where $z_x(x)=\zeta(x,0)$, 
and $s$ solves the problem
\begin{equation} \label{eq:s}
\left\{
\begin{array}{l}
s_t +F s_x= 0, \\
s(x,0) = x, \quad x \in \R. 
\end{array}
\right.
\end{equation}
The sequence of corresponding positive periodic solutions $\{g_{\infty n}\}$ 
to~\eqref{eq:per-pr}, \eqref{eq:per-mass-conservation} is bounded in 
$L^\infty(0,T;L^1(S^1))$, and $\{\fr {g_{\infty n}}{t}\}$ is bounded in 
$L^\infty(0,T;W^{-2}_1(S^1))$, so, by the Aubin--Lions--Simon lemma, without 
loss of generality, $g_{\infty n}$ converges strongly in 
$C([0,T];H^{-1}(S^1))$ to some $\zeta_\infty$.  Due to 
\eqref{eq:def-asymptotic-velocity}, for large $n$, the sign of $v_{\infty n}$ 
coincides with the sign of \begin{equation}
\label{eq:def-asymptotic-0}
v_{\infty0}
=
 \frac 1T
\int_0^T
\langle 
\zeta_\infty(\cdot, t), F(\cdot,t)\rangle_{H^{-1}(S^1)\times H^{1}(S^1)}
\cd x
\cd t,
\end{equation} provided $v_{\infty0}\neq 0$. 

Let $z_\infty\in L^{2}_{loc}(\R)$ be such that $ 
(z_\infty)_x(x)=\zeta_\infty(x,0)$ (obviously, $\zeta_\infty(\cdot,0)$ can be 
considered as a distribution on $\R$).  Since $\langle g_{\infty 
n}(\cdot,0),1\rangle=1$, in the limit we have $\langle 
\zeta_{\infty}(\cdot,0),1\rangle=1,$ so $z_\infty\big|_{[0,1]}$ is a 
non-constant function.  The functions $g_{\infty n}$ are positive by Theorem 
\ref{th:biper-sol}, so $\zeta_{\infty}(\cdot,0)\geq 0$ in the sense of 
distributions on $S^1$, whence $z_\infty$ is essentially a non-decreasing 
function. 

Passing to the limit, we see that $\zeta_\infty$ solves~\eqref{eq:per-pr0} in 
the sense of distributions.  Hence, 
\begin{equation}\zeta_\infty(x,t)=[z_\infty(s(x,t))]_x,\end{equation} and 
\begin{multline}\label{eq:z3}[z_\infty(s(x,t))]_t+F(x,t)\zeta_\infty(x,t)\\=[z_\infty(s(x,t))]_t+F(x,t)[z_\infty(s(x,t))]_x=0
\end{multline}
in the sense of distributions.  The second equality is obvious if $z_\infty$ 
is smooth, and the general case follows from the dominated convergence 
theorem.

But
\eqref{eq:def-asymptotic-0} and \eqref{eq:z3} yield 
\begin{multline} 
\label{eq:def-asymptotic-1}
v_{\infty0}=-\frac 1T
\int_0^T
\int_{0}^{1}
[z_\infty(s(x,t))]_t
\cd x
\cd t \\ =\frac {
\int_{0}^{1}(
z_\infty(x)-z_\infty(s(x,T))
\cd x
} T. 
\end{multline}
Solving \eqref{eq:s} by the method of characteristics we infer that $x=y(T)$, 
where $y(t)$ is the solution of \eqref{ode1} with the initial condition 
$y(0)=s(x,T)$.  Since $S^1$ is compact and \eqref{ode1} has no $T$-periodic 
solutions, the difference $$d(y)=y(T)-y(0)$$ is separated from zero, and its 
sign does not change for all solutions of \eqref{ode1}.  Consequently, 
$v_{\infty0}\neq 0$ (otherwise $z_\infty\big|_{[0,1]}$ would be a constant).  
We conclude that $v_{\infty 0}$ and $v_{\infty n}$ (for large $n$) have the 
same sign as $d(y)$, reaching a contradiction.
\end{proof} 

\begin{remark}
Observe that \eqref{ode1} has no $T$-periodic solutions if and only if there 
exists an unbounded trajectory or, in other words, if and only if the 
corresponding Poincar\'{e}'s rotation number \cite{cl55} is nonzero.  Hence, 
the sign of the rotation number coincides with the sign of $v_{\infty 0}$.  A 
related observation was made in \cite{sg06} for diffusion-free tilting 
ratchets.  Our conjecture is that $\lim_{\sigma\to 0}{v_\infty}=r/T$, where 
$r$ is the rotation number (cf.\ \cite{sg06}).
\end{remark}
\begin{remark}
Existence of periodic orbits is unstable with respect to perturbations of $F$, 
thus, in a general position, $v_\infty$ is non-zero for small $\sigma$.  
Unfortunately, as R.~Ortega and F.~Zanolin pointed out in a personal 
communication, there are no criteria of the form that `some set of assumptions 
on $F$ implies non-existence of $T$-periodic solutions to \eqref{ode1}'.  In 
\cite{o87}, the non-existence of periodic solutions to \eqref{ode1} with a 
tilting potential of the form $F_{sin}(y,t)=H(t)-\sin y$, where $H$ is 
$T$-periodic and satisfies \eqref{eq:unbiasedH}, was discovered.  There, it 
was observed that this phenomenon contrasts with the behaviour of the equation
\begin{equation}
\label{ode2o}
y''(t)=F_{sin}(y(t),t),
\end{equation}
which always admits $T$-periodic solutions.  However, for this particular 
class of tilting potentials there is some hope to get a criterion using a 
trick from \cite{ot00}, transforming \eqref{ode1} into a Ricatti equation.  
The discrete analogue of the problem produces Arnold tongues.
\end{remark}

\section {Multi-state models} \label{msm}

We now consider another class of ratchet models, where the particles can be in 
several states, and the total amount of particles is fixed.  Particles in 
different states are sensitive to different time-independent potentials.  The 
particles can randomly change their states with rates $\nu_{ij}(x)\geq 0$ 
(from $j$-th state to $i$-th state; $i,j=1,\dots,N$; $i\neq j$), calling forth 
the bulk transport.  This leads to the following general Cauchy problem for a 
system of Fokker--Planck-type equations:
\begin{equation}
\label{eq:general-system}
\left\{
\begin{array}{l}
(\rho_i)_t - \sigma(\rho_i)_{xx} - ((\Psi_i)_x \rho_i)_x + \sum\limits_{j,\ j\neq i}^{}\nu_{ji} \rho_i  = \sum\limits_{j,\ j\neq i}^{}\nu_{ij} \rho_j, \\
\rho_i(x,0) = \rho_{i0}(x),
\end{array}
\right.
\end{equation}
which we again consider for all $
x \in \R$ and $t > 0$.

Here $\Psi_i(x)$ are given $C^{4}$-regular potentials, ${\rho_i}_0$ are given initial mass
distributions, and $\sigma$ is the diffusion coefficient (for definiteness, we set it to be the same for all states).  We assume that $(\Psi_i)_x(x)$ and $\nu_{ij}(x)$ are $1$-periodic  and that
$\rho_{i0}(x)$ satisfy the requirements
\begin{equation}
\label{eq:ic-requirements-s}
\rho_{i0}(x) \ge 0,\ \sum\limits_{i=1}^N\int\limits_{-\infty}^{\infty} \rho_{i0} (x) \cd x = 1,\ \int\limits_{-\infty}^{\infty}|x| \rho_{i0}(x) \cd x < 
\infty.
\end{equation}

Standard arguments show that \eqref{eq:general-system} has a unique solution 
for any continuous vector function $\rho_0(x)=(\rho_{i0})$ satisfying 
\eqref{eq:ic-requirements-s}; moreover, for any $t > 0$, the components 
$\rho_i(x,t)$ of the solution vector $\rho(x,t)$ are positive.  Moreover, the 
classical estimates \cite{e69} for the fundamental matrix of solutions to 
\eqref{eq:general-system} imply
\begin{equation}
\label{eq:limit-s}
\lim_{x \to \pm\infty}
|x|(|\rho(x, t)|+|\rho_x(x,t)|)
= 0
.
\end{equation}
Let \begin{equation}
\label{eq:total-mass-s}
\tilde\rho(x,t)
=
\sum\limits_{i=1}^N\rho_i(x, t)\end{equation} be the total density function.
Adding the equations of \eqref{eq:general-system} gives
\begin{equation} \label{eq:general-sum-s}
\tilde\rho_t - \sigma\tilde\rho_{xx} - \sum\limits_{i=1}^N((\Psi_i)_x \rho_i)_x =0.
\end{equation}
As in Section \ref{sectionbv}, \eqref{eq:limit-s} and \eqref{eq:general-sum-s} yield  conservation of the total mass
\begin{equation}
\label{eq:mass-conservation-s}
\int_{-\infty}^\infty \tilde\rho(x,t) \, \dd x = 1,
\end{equation}
and finiteness of the centre of mass
\begin{equation}
\label{eq:centre-of-mass-s}
\bar x(t)
=
\int_{-\infty}^{\infty}
x \tilde\rho(x, t)
\,\dd x
\end{equation}
for any $t$.  
Similarly to Section \ref{sectionbv}, we are interested in the properties of the velocity of the 
centre of mass $\bar x(t)$.

We consider the following auxiliary problem on $S^1\times (0, \infty)$:
\begin{equation} \label{eq:per-pr-s}
\left\{
\begin{array}{l}
(g_i)_t - \sigma(g_i)_{xx} - ((\Psi_i)_x g_i)_x + \sum\limits_{j,\ j\neq i}^{}\nu_{ji} g_i  = \sum\limits_{j,\ j\neq i}^{}\nu_{ij} g_j, \\
g_i(x) > 0, \ \sum\limits_{i=1}^N\int_{S^1}g_i(x) \cd x = 1
.
\end{array}
\right.
\end{equation}

Observe that if a vector $\rho$ with positive components solves 
\eqref{eq:general-system}, then 
\begin{equation}
\label{eq:rho-g-s}
g(x, t)
=
\sum_{k = -\infty}^\infty
\rho(x + k, t)
\end{equation}
solves~\eqref{eq:per-pr-s}.  The convergence of this series follows from the 
properties of the fundamental matrix of solutions to 
\eqref{eq:general-system}. 

Then the velocity of the centre of mass is
\begin{multline}
\tilde v(t) := \fr{\bar x}{t}(t) =
\int_{-\infty}^{\infty}
x \tilde\rho_t(x, t) \cd x
\\
=\int_{-\infty}^{\infty}
x \left(\sigma \tilde\rho_{xx} + \sum\limits_{i=1}^N((\Psi_i)_x \rho_i)_x\right)
\cd x
\\
=-\int_{-\infty}^{\infty}
\left(\sigma \tilde\rho_{x} + \sum\limits_{i=1}^N(\Psi_i)_x \rho_i\right)
\cd x
\\=
-\sum\limits_{i=1}^N
\int_{0}^{1}
(\Psi_i)_x g_i
\,\dd x
.
\end{multline}

We need the following result on existence of a unique attractor 
for~\eqref{eq:per-pr-s}.

\begin{prop}
\label{th:biper-sol-s}
There exists a unique regular stationary solution vector 
$g_\infty(x)=(g_{i\infty})$ to~\eqref{eq:per-pr-s}.  Moreover, there exists 
$\gamma>0$ such that, for any smooth solution $g$ to \eqref{eq:per-pr-s}, 
there exists a positive constant $C=C(g)$ such that
\begin{equation}
\label{eq:biper-attraction-s}
| g(x, t) - g_\infty(x) | \le C \e^{-\gamma t}, \quad (x,t)\in S^1\times 
(0,\infty)
.
\end{equation}
\end{prop}

\begin{proof} The existence and uniqueness can be proved similarly to 
\cite[Theorem~4.1]{chko09}, and the stability follows in the same way as in 
\cite[Theorem~4.1]{hkl07}.
\end{proof}

We can now put
\begin{equation}
\label{eq:def-s-velocity}
\tilde v_\infty:
=
-\sum\limits_{i=1}^N
\int_{0}^{1}
(\Psi_i)_x(x) g_{i\infty}(x)
\,\dd x,
\end{equation}
arriving at the following bound.
\begin{cor}
\label{cor:asymptotic-velocity-s}
For any solution of \eqref{eq:general-system}, we have
\begin{equation}
\label{eq:asymptotic-velocity-s}
| \tilde v(t) - \tilde v_\infty | \le Ce^{-\gamma t},
\end{equation}
where $C$ depends only on $\rho_0$.

\end{cor}

Thus, the sign of the asymptotic velocity $\tilde v_\infty$ determines the 
direction of transport.  We now examine several situations when $\tilde 
v_\infty$ can be proved to be non-zero.  For simplicity, we restrict ourselves 
to the two-state models.  We begin with the small diffusion case.

\begin{theorem}
\label{odesys}
Let $N=2$.  Suppose that the functions $F_1=-(\Psi_1)_x$ and $F_2=-(\Psi_2)_x$ 
merely have a finite number of zeros on $S^1$, and do not admit common zeros.  
Let  $\sigma$ be sufficiently small, $\nu_{12}(x)> 0$, $\nu_{21}(x)> 0$. 

i) Assume that both $F_1$ and $F_2$ have zeros, and for every $x_*$ such that 
$F_1(x_*)F_2(x_*)=0$, let $i_*=1$ if $F_2(x_*)=0$, and $i_*=2$ if 
$F_1(x_*)=0$.  Then \begin{equation}\label{eq:sgn1} \operatorname{sign}\tilde 
v_\infty=\operatorname{sign} F_{i_*}(x_*), \end{equation} provided the sign in the 
right-hand side of \eqref{eq:sgn1} is independent of $x_*$.

ii) Assume that only one of the functions $F_1$ and $F_2$ possesses zeros.  
Let $F_{i_*}$ be the potential gradient which does not have zeros.  Then 
\begin{equation}
\label{eq:sgn2}
\operatorname{sign}\tilde v_\infty=\operatorname{sign} F_{i_*}.
\end{equation}  

iii) Assume that none of the functions $F_1$ and $F_2$ vanishes.  Then 
\begin{equation}\label{eq:sgn3} \operatorname{sign}\tilde 
v_\infty=\mathrm{sign}(F_1F_2)\,\operatorname{sign} \int\limits_0^1 
\left(\frac {\nu_{12}(x)} {F_2(x)}+\frac {\nu_{21}(x)}{F_1(x)}\right)\, dx, 
\end{equation} provided the integral in the right-hand side is non-zero.   
\end{theorem}

\begin{remark} In the case when $F_1$ and $F_2$ have common zeros or the sign in the right-hand side of \eqref{eq:sgn1} varies, the transport still can be present, but more subtle methods are required to establish that $\tilde v_\infty\neq 0$.
\end{remark}
\begin{remark}
In case i), if the sign in \eqref{eq:sgn1} is positive, it is simple to see 
that $\max(F_1,F_2)>0$, and there is an interval where $\min(F_1,F_2)>0$. 
Hence, the occurrence of unidirectional transport in the Neumann problem 
setting on a bounded segment follows from \cite[Theorem 1]{ps09} or 
\cite[Theorem 3.1]{hkl07}.  The cases ii) and iii) are not covered by the 
results of \cite{chk04,hkl07,ps09}.
\end{remark}

\begin{proof}[Proof of Theorem~\ref{odesys}]
 Assume there is a sequence $\sigma_n\to 0$ such that the corresponding velocities $\tilde v_{\infty n}$ violate \eqref{eq:sgn1}, or \eqref{eq:sgn2}, or \eqref{eq:sgn3}, respectively. 

The sequence of stationary solutions $\{g_{\infty n}\}$ is bounded in 
$[L^1(S^1)]^2$, so, without loss of generality, $g_{\infty n}$ converges 
weakly-$*$ in $[C^*(S^1)]^2$ to some $\zeta$.  For large $n$, the sign of 
$\tilde v_{\infty n}$ coincides with the sign of \begin{equation}
\label{eq:def-st0}
	\tilde v_{\infty0}
=
\langle 
\zeta_{1}, F_1\rangle_{C^*\times C}+\langle 
\zeta_{2}, F_2\rangle_{C^*\times C},
\end{equation} provided $\tilde v_{\infty0}\neq 0$. 
Observe that $\zeta$ solves the system
\begin{equation}
\label{eq:zetaeq}
\left\{
\begin{array}{l}
(F_1 \zeta_1)_x=\nu_{12}\zeta_2-\nu_{21}\zeta_1,
\\ (F_2 \zeta_2)_x=\nu_{21}\zeta_1-\nu_{12}\zeta_2,\\ \zeta_i\geq 0,\ \left\langle \zeta_1+\zeta_2,1\right\rangle =1
\end{array} \right.
\end{equation}
in the sense of distributions on $S^1$.  Adding these two equations, we see 
that $F_1 \zeta_1+F_2 \zeta_2$ is essentially a constant, and, due to \eqref 
{eq:def-st0},
\begin{equation}
\label{eq:def-st1}
	F_1 \zeta_1+F_2 \zeta_2=\tilde v_{\infty0}.
\end{equation}
in the sense of distributions. 

Consider first cases i) and ii).  The measures $\zeta_1$ and $\zeta_2$ cannot 
be concentrated only at the zeros of $F_1F_2$.  Indeed, if that is so, then 
\eqref{eq:def-st1} implies \begin{equation} \label{eq:def-st2}
	F_1 \zeta_1= F_2 \zeta_2 =0
\end{equation}
in the sense of distributions, so the supports of $\zeta_1$ and $\zeta_2$ are 
disjoint.  But \eqref{eq:zetaeq} gives $\nu_{12}\zeta_2=\nu_{21}\zeta_1$, so 
the supports should coincide, and we get a contradiction. 

Hence, there are two adjacent zeros $x_1$ and $x_2$ of $F_1F_2$ such that the 
support of $\zeta_1$ or $\zeta_2$ intersects with the interval $(x_1,x_2)$ (in 
the case when $F_1F_2$ has only one zero, we can change the variable $x$ for 
$\tilde x=x/2$ in \eqref{eq:general-system}, which doubles the number of zeros 
on $S^1$, but does not affect the transport direction). 

Suppose that $\tilde v_{\infty0}=0$.  Then the solution of the system 
\eqref{eq:zetaeq}, \eqref{eq:def-st1} on $(x_1,x_2)$ may be written 
explicitly, with an unknown multiplicative constant $M\neq 0$,
\begin{equation}
\label{eq:solzeta}
\left\{
\begin{array}{l}
\zeta_1(x)=M{\exp\left(\int_x^{\bar x}\frac{\nu_{12}(y)}{F_2(y)}+\frac{\nu_{21}(y)}{F_1(y)}\,dy\right)}/{F_1(x)},
\\ \zeta_2(x)=-M{\exp\left(\int_x^{\bar x}\frac{\nu_{12}(y)}{F_2(y)}+\frac{\nu_{21}(y)}{F_1(y)}\,dy\right)}/{F_2(x)},
\end{array} \right.
\end{equation}
where $\bar x=\frac {x_1 +x_2} 2$.  This gives that $F_1F_2$ is negative on 
$(x_1, x_2)$.  Consequently, $x_1$ and $x_2$ are zeros of the same 
potential --- in case ii) this is trivial --- and, for definiteness, let it be 
$F_1$.  Observe that one of the integrals $\int_{x_1}^{\bar 
x}\frac{\nu_{12}(y)}{F_2(y)}+\frac{\nu_{21}(y)}{F_1(y)}\,dy$ and 
$\int_{x_2}^{\bar x}\frac{\nu_{12}(y)}{F_2(y)}+\frac{\nu_{21}(y)}{F_1(y)}\,dy$ 
is positive.  Since $F_1$ is $C^1$-smooth, $\int_{x_1}^ {x_2}\frac 1 
{F_1(y)}\,dy=\pm\infty$.  Hence, $\int_{x_1}^ 
{x_2}\zeta_1|_{(x_1,x_2)}(y)\,dy=\infty$.  But \eqref{eq:zetaeq} implies 
$\int_{x_1}^ {x_2}\zeta_1|_{(x_1,x_2)}(y)\,dy\leq 1$, and we arrive at a 
contradiction.

Assume now that $F_1>0$ on $(x_1,x_2)$.  We then have to show that $\tilde 
v_{\infty0}< 0$.  One can check that, under the assumptions that we have made, 
the only solution of \eqref{eq:zetaeq}, \eqref{eq:def-st1} with finite 
integral on $(x_1,x_2)$ is
\begin{equation}
\label{eq:solzeta1}
\left\{
\begin{array}{l}
\zeta_1(x)=\frac{\tilde v_{\infty0}}{F_1(x)}\int_{x_1}^{x}{\exp\left(\int_x^{s}\frac{\nu_{12}(y)}{F_2(y)}+\frac{\nu_{21}(y)}{F_1(y)}\,dy\right)}
\frac {\nu_{12}(s)}{F_2(s)}\,ds,
\\ \zeta_2(x)=
\frac{\tilde v_{\infty0}}
{F_2(x)}
\left[1-\int_{x_1}^{x}
{\exp\left(\int_x^{s}\frac{\nu_{12}(y)}{F_2(y)}
+\frac{\nu_{21}(y)}{F_1(y)}\,dy\right)}
\frac {\nu_{12}(s)}{F_2(s)}\,ds
\right].
\end{array} \right.
\end{equation}
So $\zeta_1\geq 0$ and $F_2|_{(x_1,x_2)}<0$  imply $\tilde v_{\infty0}< 0$. 
The case $F_1<0$ is treated similarly. 

In case iii), the solution $\zeta$ is smooth on the whole circle $S^1$, so it 
may be identified with a $1$-periodic function on $\R$.  Moreover, it has the 
form
\begin{multline}\label{eq:solzeta2}
\zeta_1(x)=\frac M {F_1(x)}
{\exp\left(\int_x^{0}\frac{\nu_{12}(y)}
{F_2(y)}+\frac{\nu_{21}(y)}
{F_1(y)}\,dy\right)}\\+\frac{\tilde v_{\infty0}}
{F_1(x)}\int_{0}^{x}{\exp\left(\int_x^{s}\frac{\nu_{12}
(y)}{F_2(y)}+\frac{\nu_{21}(y)}{F_1(y)}\,dy\right)}
\frac {\nu_{12}(s)}{F_2(s)}\,ds,
\end{multline} \begin{multline}\label{eq:solzeta22} \zeta_2(x)=-\frac M {F_2(x)}{\exp\left(\int_x^{0}\frac{\nu_{12}(y)}{F_2(y)}+\frac{\nu_{21}(y)}{F_1(y)}\,dy\right)} \\ +\frac{\tilde v_{\infty0}}
{F_2(x)}
\left[1-\int_{0}^{x}
{\exp\left(\int_x^{s}\frac{\nu_{12}(y)}{F_2(y)}
+\frac{\nu_{21}(y)}{F_1(y)}\,dy\right)}
\frac {\nu_{12}(s)}{F_2(s)}\,ds\right].
\end{multline} 
Due to periodicity, \begin{multline}\label{eq:zetaper} M\left[ 
{\exp\left(\int_1^{0}\frac{\nu_{12}(y)}
{F_2(y)}+\frac{\nu_{21}(y)}
{F_1(y)}\,dy\right)}-1\right] \\ +
\tilde v_{\infty0}
\int_{0}^{1}{\exp\left(\int_x^{s}\frac{\nu_{12}
(y)}{F_2(y)}+\frac{\nu_{21}(y)}{F_1(y)}\,dy\right)}
\frac {\nu_{12}(s)}{F_2(s)}\,ds=0.\end{multline} Hence, $\tilde v_{\infty0}\neq 0$, for the contrary would imply $M=0$, and $\zeta\equiv 0$, which contradicts \eqref{eq:zetaeq}. 

Assume that the right-hand side of \eqref{eq:sgn3} is positive, and $\tilde 
v_{\infty0}$ is negative (the opposite situation is completely analogous).  
Then at least one of the potentials, say $F_1$, is positive.  Due to the 
positivity of the right-hand side of~\eqref{eq:sgn3},
\begin{equation}
\operatorname{sign}\left( {\exp\left(\int_1^{0}\frac{\nu_{12}(y)}
{F_2(y)}+\frac{\nu_{21}(y)}
{F_1(y)}\,dy\right)}-1\right)=-\operatorname{sign}F_2
,
\end{equation}
and \eqref{eq:zetaper} yields
\begin{equation}
\label{eq:signm}
\operatorname{sign} M = \operatorname{sign} \tilde v_{\infty0}.
\end{equation}
It follows from \eqref{eq:solzeta2} and \eqref{eq:signm} that $\zeta_1(0)<0$, 
which contradicts the positivity of the solution.
\end{proof}

We now address the transport properties of the \emph{randomly tilting 
ratchet}, i.e., the model of the form
\begin{equation} \label{eq:rt-system}
\left\{
\begin{array}{l}
(\rho_1)_t - (\rho_1)_{xx} - ([\psi_x +\omega]\rho_1)_x + \nu_{21} \rho_1  =\nu_{12} \rho_2, \\
(\rho_2)_t - (\rho_2)_{xx} - ([\psi_x +\Omega] \rho_2)_x + \nu_{12} \rho_2  = \nu_{21} \rho_1, \\
\rho(x,0) = \rho_{0}(x),
\end{array}
\right.
\end{equation}
where $\psi(x)$ is a $C^4$-smooth $1$-periodic potential, the diffusion 
coefficient $\sigma$ is taken to be $1$; $\omega, \Omega, \nu_{12} > 0, 
\nu_{21}>0$ are scalars (independent of $x$, for simplicity), and $\rho_0$ 
satisfies~\eqref{eq:ic-requirements-s}.  We assume the following non-bias 
condition (cf.~\eqref{eq:semiad-cond-unbiased}):
\begin{equation}
\label{eq:rt-unbiased}
\Omega \nu_{21} + \omega\nu_{12} = 0.
\end{equation}
Denote by $\tilde v_{\infty}(\omega,\nu)$ the corresponding bulk velocity defined by \eqref{eq:def-s-velocity}. 

The following theorem shows that the adiabatic and semi-adiabatic bulk 
velocities of the randomly tilting ratchet are the same as for the tilting 
ratchet.  Thus, the results of the previous sections may be applied to 
determine the transport direction.  In particular, the direction of the 
semiadiabatic transport is determined by the sign of $\omega$ for every 
non-constant $\psi$. 

\begin{theorem} \label{rtasa}

Let $\omega$ be fixed.  Then
\begin{equation}
\lim_{\nu_{12} = \nu_{21} \to 0} \tilde v_{\infty}(\omega,\nu) = - \frac 
{A(\omega) + A(-\omega)}2
;
\end{equation}
\begin{equation}
\lim_{
\substack{
\nu_{21} \to 0 \\
\nu_{21}/\nu_{12} \to 0
}}
\tilde v_{\infty}(\omega,\nu)
=\omega-A(\omega)
.
\end{equation}
\end{theorem}
The proof is left as an exercise for the reader.

\bibliography{ratchets_kuv}

\def\cprime{$'$}
\begin{thebibliography}{10}

\bibitem{ah03}
R.~Ait-Haddou and W.~Herzog.
\newblock Brownian ratchet models of molecular motors.
\newblock {\em Cell biochemistry and biophysics}, 38(2):191--213, 2003.

\bibitem{as97}
R.~D. Astumian.
\newblock Thermodynamics and kinetics of a brownian motor.
\newblock {\em Science}, 276(5314):917--922, 1997.

\bibitem{bdik07}
J.-P. Bartier, J.~Dolbeault, R.~Illner, and M.~Kowalczyk.
\newblock A qualitative study of linear drift-diffusion equations with
  time-dependent or degenerate coefficients.
\newblock {\em Mathematical Models and Methods in Applied Sciences},
  17(03):327--362, 2007.

\bibitem{bena00}
I.~Bena, M.~Copelli, and C.~Van~den Broeck.
\newblock Stokes' drift: A rocking ratchet.
\newblock {\em Journal of Statistical Physics}, 101(1-2):415--424, 2000.

\bibitem{bdk08}
A.~Blanchet, J.~Dolbeault, and M.~Kowalczyk.
\newblock Travelling fronts in stochastic {S}tokes' drifts.
\newblock {\em Physica A: Statistical Mechanics and its Applications},
  387(23):5741--5751, 2008.

\bibitem{bdk09}
A.~Blanchet, J.~Dolbeault, and M.~Kowalczyk.
\newblock Stochastic {S}tokes' drift, homogenized functional inequalities, and
  large time behavior of {B}rownian ratchets.
\newblock {\em SIAM J. Math. Anal.}, 41(1):46--76, 2009.

\bibitem{bf06}
W.~R. Browne and B.~L. Feringa.
\newblock Making molecular machines work.
\newblock {\em Nature nanotechnology}, 1(1):25--35, 2006.

\bibitem{chk04}
M.~Chipot, S.~Hastings, and D.~Kinderlehrer.
\newblock Transport in a molecular motor system.
\newblock {\em ESAIM: Mathematical Modelling and Numerical Analysis},
  38(06):1011--1034, 2004.

\bibitem{chko09}
M.~Chipot, D.~Hilhorst, D.~Kinderlehrer, and M.~Olech.
\newblock Contraction in {$L^1$} for a system arising in chemical reactions and
  molecular motors.
\newblock {\em Differ. Equ. Appl.}, 1(1):139--151, 2009.

\bibitem{cl55}
E.~A. Coddington and N.~Levinson.
\newblock {\em Theory of ordinary differential equations}.
\newblock McGraw-Hill Book Company, Inc., New York-Toronto-London, 1955.

\bibitem{csi75}
I.~Csisz{\'a}r.
\newblock {$I$}-divergence geometry of probability distributions and
  minimization problems.
\newblock {\em Ann. Probability}, 3:146--158, 1975.

\bibitem{da11}
A.-L. Dalibard.
\newblock Stability of periodic stationary solutions of scalar conservation
  laws with space-periodic flux.
\newblock {\em Journal of the European Mathematical Society}, 13(5):1245--1288,
  2011.

\bibitem{dkk04}
J.~Dolbeault, D.~Kinderlehrer, and M.~Kowalczyk.
\newblock Remarks about the flashing ratchet.
\newblock In {\em Partial differential equations and inverse problems}, volume
  362 of {\em Contemp. Math.}, pages 167--175. Amer. Math. Soc., Providence,
  RI, 2004.

\bibitem{e69}
S.~D. {E}idelman.
\newblock {\em Parabolic systems}.
\newblock North Holland Pub. Co., 1969.

\bibitem{fi02}
M.~Fistul.
\newblock Symmetry broken motion of a periodically driven brownian particle:
  Nonadiabatic regime.
\newblock {\em Physical Review E}, 65(4):046621, 2002.

\bibitem{gori}
M.~Gori.
\newblock {\em Lower semicontinuity and relaxation for integral and supremal
  functionals}.
\newblock PhD thesis, University of Pisa, Apr. 2004.

\bibitem{go09}
N.~Grunewald, F.~Otto, C.~Villani, and M.~G. Westdickenberg.
\newblock A two-scale approach to logarithmic {S}obolev inequalities and the
  hydrodynamic limit.
\newblock {\em Ann. Inst. Henri Poincar\'e Probab. Stat.}, 45(2):302--351,
  2009.

\bibitem{hm09}
P.~H{\"a}nggi and F.~Marchesoni.
\newblock Artificial brownian motors: Controlling transport on the nanoscale.
\newblock {\em Reviews of Modern Physics}, 81(1):387, 2009.

\bibitem{hmn05}
P.~H{\"a}nggi, F.~Marchesoni, and F.~Nori.
\newblock Brownian motors.
\newblock {\em Annalen der Physik}, 14(1-3):51--70, 2005.

\bibitem{hkl07}
S.~Hastings, D.~Kinderlehrer, and J.~B. McLeod.
\newblock Diffusion mediated transport in multiple state systems.
\newblock {\em SIAM J. Math. Anal.}, 39(4):1208--1230, 2007.

\bibitem{jl98}
K.~M. Jansons and G.~Lythe.
\newblock Stochastic {S}tokes drift.
\newblock {\em Physical Review Letters}, 81(15):3136, 1998.

\bibitem{klz07}
E.~Kay, D.~Leigh, and F.~Zerbetto.
\newblock Synthetic molecular motors and mechanical machines.
\newblock {\em Angewandte Chemie International Edition}, 46(1-2):72--191, 2007.

\bibitem{kk02}
D.~Kinderlehrer and M.~Kowalczyk.
\newblock Diffusion-mediated transport and the flashing ratchet.
\newblock {\em Arch. Ration. Mech. Anal.}, 161(2):149--179, 2002.

\bibitem{ms13}
S.~Mirrahimi and P.~E. Souganidis.
\newblock A homogenization approach for the motion of motor proteins.
\newblock {\em Nonlinear Differential Equations and Applications NoDEA},
  20(1):129--147, 2013.

\bibitem{o87}
R.~Ortega.
\newblock A counterexample for the damped pendulum equation.
\newblock {\em Acad. Roy. Belg. Bull. Cl. Sci. (5)}, 73(10):405--409, 1987.

\bibitem{ot00}
R.~Ortega and M.~Tarallo.
\newblock Degenerate equations of pendulum-type.
\newblock {\em Commun. Contemp. Math.}, 2(2):127--149, 2000.

\bibitem{par02}
J.~Parrondo and B.~J. de~Cisneros.
\newblock Energetics of brownian motors: a review.
\newblock {\em Applied Physics A}, 75(2):179--191, 2002.

\bibitem{ps09a}
B.~Perthame and P.~E. Souganidis.
\newblock Asymmetric potentials and motor effect: a homogenization approach.
\newblock {\em Ann. Inst. H. Poincar\'e Anal. Non Lin\'eaire},
  26(6):2055--2071, 2009.

\bibitem{ps09}
B.~Perthame and P.~E. Souganidis.
\newblock Asymmetric potentials and motor effect: a large deviation approach.
\newblock {\em Arch. Ration. Mech. Anal.}, 193(1):153--169, 2009.

\bibitem{ps11}
B.~Perthame and P.~E. Souganidis.
\newblock A homogenization approach to flashing ratchets.
\newblock {\em Nonlinear Differential Equations and Applications NoDEA},
  18(1):45--58, 2011.

\bibitem{pe84}
F.~O. Porper and S.~D. {\`E}{\u\i}del{\cprime}man.
\newblock Two-sided estimates of the fundamental solutions of second-order
  parabolic equations and some applications of them.
\newblock {\em Uspekhi Mat. Nauk}, 39(3(237)):107--156, 1984.

\bibitem{rei02}
P.~Reimann.
\newblock Brownian motors: noisy transport far from equilibrium.
\newblock {\em Physics Reports}, 361(2):57--265, 2002.

\bibitem{rh02}
P.~Reimann and P.~H{\"a}nggi.
\newblock Introduction to the physics of brownian motors.
\newblock {\em Applied Physics A}, 75(2):169--178, 2002.

\bibitem{sg06}
R.~Salgado-Garc{\'\i}a, M.~Aldana, and G.~Mart{\'\i}nez-Mekler.
\newblock Deterministic ratchets, circle maps, and current reversals.
\newblock {\em Physical review letters}, 96(13):134101, 2006.

\bibitem{ser61}
J.~Serrin.
\newblock On the definition and properties of certain variational integrals.
\newblock {\em Trans. Amer. Math. Soc.}, 101:139--167, 1961.

\bibitem{vor11}
D.~Vorotnikov.
\newblock The flashing ratchet and unidirectional transport of matter.
\newblock {\em Discrete Contin. Dyn. Syst. Ser. B}, 16(3):963--971, 2011.

\bibitem{vor13}
D.~Vorotnikov.
\newblock Analytical aspects of the {B}rownian motor effect in randomly
  flashing ratchets.
\newblock {\em J. Math. Biol.}, 68(7):1677--1705, 2014.

\end{thebibliography}

\bibliographystyle{abbrv}

\end{document}